\documentclass[11pt,a4paper]{siamart190516}

\usepackage[utf8]{inputenc}
\usepackage[english]{babel}

\usepackage{tikz}
\usetikzlibrary{shapes.geometric}

\usepackage{graphicx}
\usepackage{fancyhdr}
\usepackage{dsfont}
\usepackage{stmaryrd}
\usepackage{hyperref}
\usepackage{bbm}
\usepackage{comment} 

\usepackage{amsfonts}
\usepackage{amssymb}

\usepackage{amsmath}

\usepackage{amsfonts}
\usepackage{amssymb}
\usepackage{mathtools}
\usepackage{enumitem}
\usepackage{tikz} 

\usepackage{xcolor}

\definecolor{cinnamon}{rgb}{0.82, 0.41, 0.12}

\definecolor{ao(english)}{rgb}{0.0, 0.5, 0.0}

\usepackage{mathrsfs}

\usepackage{overpic} 

\usepackage{algorithm}
\usepackage[noend]{algpseudocode}
\algblock{ParFor}{EndParFor}
\algnewcommand\algorithmicparfor{\textbf{parallel for}}
\algnewcommand\algorithmicpardo{\textbf{do}}
\algrenewtext{ParFor}[1]{\algorithmicparfor\ #1\ \algorithmicpardo}
\algtext*{EndParFor}

\usepackage{mathtools,booktabs}
\DeclarePairedDelimiter\ceil{\lceil}{\rceil}

\algrenewcommand\algorithmicrequire{\textbf{Input:}}
\algrenewcommand\algorithmicensure{\textbf{Output:}}

\usepackage{varwidth}

\usepackage[title]{appendix}

\newcommand{\N}{\mathbb{N}}
\newcommand{\TT}{\mathbb{T}}

\newcommand{\norm}[1]{\left\lVert#1\right\rVert}
\newcommand{\abs}[1]{\left|#1\right|}

\newcommand{\R}[0]{\mathbb{R}}
\newcommand{\C}[0]{\mathbb{C}}
\newcommand{\Nmin}{N^{\rm neg}_\alpha}

\newcommand{\modifyadd}[1]{\textcolor{black}{#1}}

\usepackage{cite}
\usepackage{color}

\title{On the stability of unevenly spaced samples for interpolation and quadrature\thanks{Submitted to the editors \today.
\funding{This work is partially supported by the National Science Foundation grants DMS-1818757, DMS-1952757 and DMS-2045646.}}}
\author{Annan Yu\thanks{Center for Applied Mathematics, Cornell University, Ithaca, NY 14853. (\email{ay262@cornell.edu})}
\and Alex Townsend\thanks{Department of Mathematics, Cornell University, Ithaca, NY 14853. (\email{townsend@cornell.edu})}}
\headers{}{A. Yu and A. Townsend}

\begin{document}
\maketitle

\begin{abstract}
Unevenly spaced samples from a periodic function are common in signal processing and can often be viewed as a perturbed equally spaced grid. In this paper, we analyze how the uneven distribution of the samples impacts the quality of interpolation and quadrature. Starting with equally spaced nodes on $[-\pi,\pi)$ with grid spacing $h$, suppose the unevenly spaced nodes are obtained by perturbing each uniform node by an arbitrary amount $\leq \alpha h$, where $0\leq \alpha< 1/2$ is a fixed constant.  We prove a discrete version of the Kadec-1/4 theorem, which states that the nonuniform discrete Fourier transform associated with perturbed nodes has a bounded condition number independent of $h$, for any $\alpha<1/4$. We go on to show that unevenly spaced quadrature rules converge for all continuous functions and interpolants converge uniformly for all differentiable functions whose derivative has bounded variation when $0\leq \alpha<1/4$. Though, quadrature rules at perturbed nodes can have negative weights for any $\alpha>0$, we provide a bound on the absolute sum of the quadrature weights. Therefore, we show that perturbed equally spaced grids with small $\alpha$ can be used without numerical woes.  \modifyadd{While our proof techniques work primarily when $0 \leq \alpha < 1/4$}, we show that a small amount of oversampling \modifyadd{extends our results to the case} when $1/4\leq \alpha<1/2$.
\end{abstract}



\section{Introduction}\label{sec:introduction}
In signal processing, function approximation, and econometrics, unevenly spaced time series data naturally occur. For example, natural disasters occur at irregular time intervals~\cite{quan}, observational astronomy takes measurements of celestial bodies at times determined by cloud coverage and planetary configurations~\cite{vio}, clinical trials may monitor health diagnostics at irregular time intervals~\cite{stahl}, and wireless sensors only record information when a state changes to conserve battery life~\cite{gowrishankar}. In most applications, the samples can usually be considered as obtained from perturbed equally spaced nodes. Therefore, a common approach to deal with unevenly spaced samples is to first transform the data into equally spaced observations using some form of low-order interpolation~\cite{eckner2012algorithms}. However, transforming data in this way can introduce a number of significant and hard-to-quantify biases~\cite{scholes1977estimating,rehfeld2011comparison}.  Ideally, unevenly spaced time series are analyzed in the original form and there seems to be limited theoretical results on this in approximation theory~\cite{austintrefethen}. 


Suppose that there is an unknown $2\pi$-periodic function $f:[-\pi,\pi)\rightarrow \mathbb{C}$ that is sampled at $2N+1$ unevenly spaced nodes, and one would like to recover $f$ via interpolation or compute integrals involving $f$. How much does the uneven distribution of the samples impact the quality of interpolation or quadrature?  To make progress on this question, we assume that the unevenly spaced nodes can be viewed as perturbed equally spaced nodes. That is, we have acquired samples $f_{-N},\ldots,f_{N}$ from $f$ at nodes that are perturbed from equally spaced nodes, i.e.,
\begin{equation} 
f_j = f(\tilde{x}_j), \qquad \tilde{x}_j = (j + \delta_j)h  \qquad -N \leq j \leq N,
\label{eq:UnevenSamples} 
\end{equation} 
where $h = 2\pi/(2N+1)$ is the grid spacing and $\delta_j$ is the perturbation of $jh$ with $|\delta_j|\leq \alpha$ for some $0 \leq \alpha < 1/2$ (see~\cref{fig.perturb}). We call the nodes $\tilde{x}_{-N},\ldots,\tilde{x}_N$ a set of $\alpha$-perturbed nodes. Here, we assume that $\alpha < 1/2$ so that nodes cannot coalesce.  When $\alpha = 0$ the nodes are equally spaced and we are in a classical setting of approximation theory. In particular, when $\alpha=0$, one can use the fast Fourier transform (FFT)~\cite{cooley1965algorithm} to compute interpolants that converge rapidly to $f$~\cite{wright2015extension}. Moreover, the associated quadrature estimate is the trapezoidal rule, which is numerically stable and can be geometrically convergent for computing integrals involving $f$~\cite{hunter,trefethenweide}. Surprisingly, there has been far less theoretical attention on the case when $\alpha\neq 0$, despite it appearing in numerous applications and many encouraging numerical observations~\cite{austin,trefethenweide}.  A notable exception is Austin and Trefethen's work~\cite{austintrefethen}, where they showed that interpolants and the quadrature at perturbed grids converge when the underlying function $f$ is twice continuously differentiable and $\alpha<1/2$. Using a discrete analogue of the Kadec-1/4 theorem~\cite{kadec}, we strengthen these results when $\alpha<1/4$. 

\begin{figure}
\label{fig.perturb}
\centering 
\begin{tikzpicture} 
\draw[black,thick] (0,0)--(9,0); 
\draw[black,thick] (0,-.2)--(0,.2); 
\filldraw (.5,0) circle (2pt);
\filldraw (1.5,0) circle (2pt);
\filldraw (2.5,0) circle (2pt);
\filldraw (3.5,0) circle (2pt);
\filldraw (4.5,0) circle (2pt);
\filldraw (5.5,0) circle (2pt);
\filldraw (6.5,0) circle (2pt);
\filldraw (7.5,0) circle (2pt);
\filldraw (8.5,0) circle (2pt);
\draw[black,thick] (9,-.25) arc (-15:15:1);
\draw[red,ultra thick] (.5-0.333,0)--(.5+0.333,0); 
\draw[red,ultra thick] (1.5-0.333,0)--(1.5+0.333,0); 
\draw[red,ultra thick] (2.5-0.333,0)--(2.5+0.333,0); 
\draw[red,ultra thick] (3.5-0.333,0)--(3.5+0.333,0); 
\draw[red,ultra thick] (4.5-0.333,0)--(4.5+0.333,0); 
\draw[red,ultra thick] (5.5-0.333,0)--(5.5+0.333,0); 
 \draw[red,ultra thick] (6.5-0.333,0)--(6.5+0.333,0); 
\draw[red,ultra thick] (7.5-0.333,0)--(7.5+0.333,0); 
\draw[red,ultra thick] (8.5-0.333,0)--(8.5+0.333,0); 
\node at (0,-.5) {$-\pi$};
\node at (9,-.5) {$\pi$};
\node at (4.5,-.3) {$0$};
\node at (5.5,-.3) {$h$};
\node at (6.5,-.3) {$2h$};
\node at (7.5,-.3) {$3h$};
\node at (8.5,-.3) {$4h$};
\node at (3.5,-.3) {$-h$};
\node at (2.5,-.3) {$-2h$};
\node at (1.5,-.3) {$-3h$};
\node at (.5,-.3) {$-4h$};
\node at (4.5,.5) {$h=\frac{2\pi}{9}$};
\end{tikzpicture} 
\caption{Nine equally spaced nodes (black dots) on $[-\pi,\pi)$. The intervals (red) show where the unevenly spaced function samples can be when $\alpha =1/3$ in~\cref{eq:UnevenSamples}.}
\end{figure}
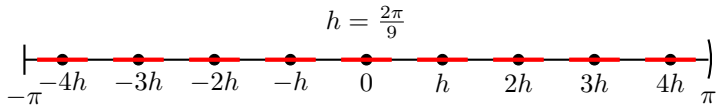


There are several aspects of unevenly spaced samples that we investigate: 

\begin{itemize}[leftmargin=*]
\item {\em Conditioning of a nonuniform discrete Fourier transform.}
The analogue of the FFT for unevenly spaced nodes is the nonuniform discrete Fourier transform (NUDFT). There are many variants of the NUDFT, but we focus on the one that is closely related to~\cref{eq:UnevenSamples} and common in signal processing~\cite{bagchi2012nonuniform}.  Let $N$ be an integer and $\underline{c} = \left(c_{-N},\ldots,c_{N}\right)^\top$ be a vector. The NUDFT task related to~\cref{eq:UnevenSamples} is to compute the vector $\underline{f} =  \left(f_{-N},\ldots,f_{N}\right)^\top$, defined by the following sums: 
\begin{equation} 
f_j = \sum_{k=-N}^{N} c_k e^{-i \tilde{x}_j k}, \qquad -N\leq j\leq N.
\label{eq:NUDFT}
\end{equation} 
As~\cref{eq:NUDFT} involves $2N+1$ sums each with $2N+1$ terms, the naive algorithms for computing $\underline{f}$ require $\mathcal{O}(N^2)$ operations; however, there are efficient algorithms that require only $\mathcal{O}(N\log N)$ operations~\cite{barnett2019parallel}. The NUDFT task has an inverse for any perturbed grid with $0\leq \alpha<1/2$, which aims to recover the vector $\underline{c}$ from $\underline{f}$. We show that for $\alpha<1/4$, the NUDFT task and its inverse have a condition number that can be bounded independent of $N$ when $\alpha<1/4$ (see~\cref{cor:NUDFTconditioning}). 

\item {\em Trigonometric interpolation.} 
A trigonometric polynomial $q$ of degree $\leq n$ is a function defined on $[-\pi,\pi)$ of the form 
\begin{equation}
	q(x) = \sum_{k=-n}^n c_k e^{ikx}, \qquad c_k \in \C.
	\label{eq:trigpoly}
\end{equation}
We denote the space of all trigonometric polynomials of degree $\leq n$ by $\mathcal{T}_n$. The goal of trigonometric interpolation is to find coefficients $c_{-n},\ldots,c_{n}$ such that $q$ interpolates $f$ at the samples, i.e., $q(\tilde{x}_j) = f_j$ for $-N\leq j\leq N$. Since the summand in~\cref{eq:trigpoly} can be reversed, an algorithm to find the coefficients is the inverse NUDFT when $n = N$. We show that interpolants at $\alpha$-perturbed nodes converge to $f$ for all differentiable functions whose derivative has bounded variation when $0\leq \alpha<1/4$ (see~\cref{thm.interpolatemax} and~\cref{thm.interpolateoptimal}). 
 
\item {\em Exact quadrature.}
A quadrature rule is a method for numerical integration that approximates the integral of $f$ by a weighted sum of the function's samples, i.e., 
\begin{equation} 
I = \int_{-\pi}^\pi f(x) dx \approx \tilde{I}_N=  \sum_{j=-N}^{N} \tilde{w}_j f(\tilde{x}_j).
\label{eq:quadratureRule} 
\end{equation} 
Here, $\tilde{x}_{-N},\ldots,\tilde{x}_N$ are called the quadrature nodes and $\tilde{w}_{-N},\ldots,\tilde{w}_N$ the quadrature weights. Given quadrature nodes, it is often desirable to design the quadrature weights so that~\cref{eq:quadratureRule} is exact for all trigonometric polynomials in $\mathcal{T}_n$ for some $n$. When $n = N$, the exact condition uniquely defines a quadrature rule for any $\alpha<1/2$. While quadrature rules at $\alpha$-perturbed nodes can have negative weights for any $\alpha>0$ (see~\cref{thm.negative}), we show that the absolute sums of the weights, i.e., $\sum_{j=-N}^N |\tilde{w}_j|$, is bounded independent of $N$ for $\alpha<1/4$ (see~\cref{thm.bounded}), which shows that the quadrature rules are numerically stable. We provide an explicit upper bound on $|I-\tilde{I}_N|$ and conclude that $\lim_{N\rightarrow\infty} \tilde{I}_N = I$ for all continuous periodic functions when $\alpha<1/4$.


\item {\em Marcinkiewicz--Zygmund inequalities.}
The stability of quadrature and interpolation is closely connected to so-called Marcinkiewicz--Zygmund (MZ) inequalities~\cite{mhaskar, grochenig}. We show that the following MZ inequality holds for all $q\in\mathcal{T}_N$: 
\begin{equation} 
	\frac{(1-\varphi_\alpha)^2}{2\pi}\! \int_{-\pi}^\pi |q(x)|^2 dx \leq \frac{1}{2N+1} \sum_{j=-N}^{N} \abs{q(\tilde{x}_j)}^2 \leq  \frac{(1+\varphi_\alpha)^2}{2\pi}\!\int_{-\pi}^\pi |q(x)|^2 dx,
\label{eq.MZp} 
\end{equation} 
where $\tilde{x}_{-N},\ldots,\tilde{x}_N$ are any $\alpha$-perturbed nodes in~\cref{eq:UnevenSamples} with $\alpha<1/4$ and $\varphi_\alpha = 1-\cos(\pi\alpha)+\sin(\pi\alpha)$. Other MZ inequalities at perturbed nodes are found in~\cite{marzoseip,ortega}. We use the MZ inequality in~\cref{eq.MZp} to derive explicit error bounds and the rate of convergence of quadrature rules and interpolants at $\alpha$-perturbed nodes when $0\leq \alpha < 1/4$. 
\end{itemize} 

\modifyadd{While our proof techniques work primarily for $0 \leq \alpha < 1/4$} when $N = n$ (see~\cref{sec:MZbounded}), we show that oversampling, i.e., $N>n$, \modifyadd{allows us to extend our results to the case when $1/4\leq \alpha <1/2$}. By oversampling by a factor of $1+\epsilon$, i.e., $N = \lceil (1+\epsilon)n\rceil$ for any $\epsilon>0$, we show that the same convergence results for interpolation and quadrature carry over to when $1/4\leq \alpha <1/2$ (see~\cref{sec:oversample}), improving on a result by Mhaskar, Narcowich, and Ward~\cite[Cor.~4.1]{mhaskar}.

\modifyadd{Many results in this paper are motivated by Austin's thesis~\cite{austin}. In particular, for $0 \leq \alpha < 1/4$, we prove Conjecture~3.10 of~\cite{austin} on the $2$-norm Lebesgue constant (see~\cref{cor.MZkadec}) and Conjecture~3.14 on the absolute sum of the quadrature weights (see~\cref{thm.bounded}). Our results for interpolation also confirm Conjecture 3.7. Moreover, we provide an answer to questions regarding the signs of quadrature weights raised in section~3.5.2 of~\cite{austin} (see~\cref{thm.negative} and~\cref{thm.thetaN}).}

The paper is structured as follows. In~\cref{sec:kadec}, we prove a discrete version of the Kadec-$1/4$ theorem that leads to a condition number bound on the nonuniform discrete Fourier transform and discuss its consequences. In~\cref{sec:interpolation} and~\cref{sec:quad}, we study the interpolation and quadrature rules at $\alpha$-perturbed nodes. In~\cref{sec:MZbounded} we look at further consequences of MZ inequalities. Finally, in~\cref{sec:oversample}, we investigate what happens when in the oversampling setting. 

\section{The Kadec-$\mathbf{1/4}$ theorem and its consequences}\label{sec:kadec}
In sampling theory, the Kadec-$1/4$ theorem shows that the Fourier modes $\left\{e^{i\lambda_k x}\right\}$ for $k\in\mathbb{Z}$ form a Riesz basis when $|\lambda_k-k|\leq \alpha$ and $0\leq \alpha<1/4$~\cite{kadec}, which in signal processing means that one can recover a square-integrable function if given its inner-products with $\left\{e^{i\lambda_k x}\right\}$. 
While the standard setting for Kadec's theorem applies to perturbing an infinite number of Fourier wave numbers from integers $k$ to $\lambda_k$, we show that there is a discrete analogue when perturbing a finite number of nodes. It has consequences for the condition number of the NUDFT and MZ inequalities. 

For an integer $N$ and $\alpha$-perturbed nodes $\{\tilde{x}_j\}$, let $F$ be the DFT matrix given by 
\[
F_{jk} = e^{-2\pi i jk/(2N+1)}, \qquad -N\leq j,k\leq N,
\]
and $\tilde{F}$ be a NUDFT matrix given by 
\begin{equation} 
\tilde{F}_{jk} = e^{-i \tilde{x}_j k} = e^{-2\pi i (j + \delta_j) k/(2N+1)}, \qquad -N\leq j,k\leq N. 
\label{eq:NUDFTmatrix}
\end{equation} 
The matrix $\tilde{F}$ is of interest because the sums in~\cref{eq:NUDFT} can be neatly written as the matrix-vector product $\tilde{F}\underline{c} = \underline{f}$. Therefore, the action of $\tilde{F}$ onto a vector is equivalent to evaluating a trigonometric polynomial of degree $N$ at $\tilde{x}_{-N},\ldots,\tilde{x}_N$. This can be performed in $\mathcal{O}(N\log N)$ operations~\cite{ruiz2018nonuniform} and in Chebfun (a MATLAB package for computing with functions)~\cite{driscoll2014chebfun} can be implemented as follows:
\begin{verbatim} 
 N = 1e4; alpha = 0.1; h = 2*pi/(2*N+1);
 tilde_x = (-N:N)'*h + alpha*(2*rand(2*N+1,1)-1)*h; % perturbed nodes
 c = randn(2*N+1,1);                                % trig coeffs
 Fc = exp(1i*tilde_x*N).*chebfun.nufft(c,tilde_x/(2*pi)); %\tilde{F}*c
\end{verbatim} 
An alternative algorithm that has even faster execution times is available in the fiNUFFT package~\cite{barnett2019parallel}. 

We now show that $\tilde{F}$ and $F$ are relatively close in the sense of their distance in the spectral norm. Our proof follows from a similar strategy to the proof of Kadec's theorem~\cite{young}. 


\begin{theorem}\label{thm.kadec}
Suppose that $\abs{\delta_j} \leq \alpha <1/4$ for all $-N\leq j\leq N$ in~\cref{eq:NUDFTmatrix}, then
\begin{equation}\label{eq.Kadec}
    \left\|F - \tilde{F}\right\|_2 \leq \varphi_\alpha\|F\|_2, \qquad \varphi_\alpha = 1-\cos(\pi\alpha)+\sin(\pi\alpha),
\end{equation}
where $\|\cdot\|_2$ is the spectral norm.
\end{theorem}
\begin{proof}
Since $\|F - \tilde{F}\|_2 = \|F^\top - \tilde{F}^\top\|_2$ and $\|F\|_2 = \|F^\top\|_2$, where superscript `$\top$' denotes the matrix transpose, we prove that $\|F^\top - \tilde{F}^\top\|_2\leq \varphi_\alpha \|F^\top\|_2$. Let $\underline{c} = (c_{-N}, \ldots, c_{N})^\top$ be a vector of unit length so that $\|\underline{c}\|_2 = 1$. We have
\begin{equation}
 (F^\top-\tilde{F}^\top)\underline{c} = \sum_{j=-N}^{N} e^{-ij\underline{t}} \circ \left(\underline{\mathbbm{1}} - e^{-i\delta_j\underline{t}}\right)c_j,
 \label{eq:NUDFTdifference}
\end{equation}
where $\underline{t} = (t_{-N},\ldots,t_{N})^\top$ with $t_k = 2\pi k/(2N+1)$, `$\circ$' is the Hadamard product denoting entry-by-entry multiplication between vectors, $\underline{\mathbbm{1}}$ is the column vector of all ones, and the exponential function of a vector is applied entrywise. As in~\cite{young}, for each $-N\leq j \leq N$, we write
\begin{align*}
    \underline{\mathbbm{1}} - e^{-i\delta_j\underline{t}} &= \underbrace{\left(1 - \frac{\sin (\pi\delta_j)}{\pi\delta_j}\right)\underline{\mathbbm{1}}}_{A_j} + \underbrace{\sum_{\ell=1}^\infty \frac{(-1)^\ell 2\delta_j \sin (\pi\delta_j)}{\pi(\ell^2 - \delta_j^2)} \cos(\ell\underline{t})}_{B_j} \\
    &\qquad -\underbrace{i\sum_{\ell=1}^\infty \frac{(-1)^\ell 2\delta_j \cos (\pi\delta_j)}{\pi(\ell - 1/2)^2 - \pi\delta_j^2} \sin\left(\left(\ell-\frac{1}{2}\right)\underline{t}\right)}_{C_j},
\end{align*}
where $\sin$ and $\cos$ of a vector is applied entrywise. By substituting this expression into~\cref{eq:NUDFTdifference}, we can bound $\| (F^\top-\tilde{F}^\top)\underline{c} \|_2$ by separately bounding $\sum_{j=-N}^{N} A_j\circ e^{-ij\underline{t}}c_j, \sum_{j=-N}^{N} B_j\circ e^{-ij\underline{t}}c_j,$ and $\sum_{j=-N}^{N} C_j\circ e^{-ij\underline{t}}c_j$.

\noindent \textbf{Bounding $\mathbf{\sum_{j=-N}^{N} A_j\circ e^{-ij\underline{t}}c_j}$.} Since $\abs{\delta_j} \leq \alpha < 1/4$, from elementary calculus, we have that $\max_j \left(1 - \sin(\pi\delta_j) / (\pi\delta_j)\right) \leq 1 - \sin(\pi\alpha)/(\pi\alpha)$. Also, since $\{e^{-ij\underline{t}}\}_{j=-N}^{N}$ is an orthogonal set and $A_j$ is a constant vector, $\{A_j \circ e^{-ij\underline{t}}c_j\}_{j=-N}^{N}$ is a set of orthogonal vectors. Hence, by the Pythagorean Theorem, we have
\begin{equation}\label{eq.controlA}
\begin{aligned}
    &\norm{\sum_{j=-N}^{N} A_j \circ e^{-ij\underline{t}}c_j}_2^2\!= \!\! \sum_{j=-N}^{N} \!\!\norm{A_j \circ e^{-ij\underline{t}}c_j}_2^2 \leq\! \left(1-\frac{\sin(\pi\alpha)}{\pi\alpha}\right)^2 \!\!\!\!\sum_{j=-N}^{N} \!\!\norm{e^{-ij\underline{t}}c_j}_2^2 \\
    &\qquad\qquad\qquad =\left(1-\frac{\sin(\pi\alpha)}{\pi\alpha}\right)^2 \norm{\sum_{j=-N}^{N} e^{-ij\underline{t}}c_j}_2^2 \leq \left(1-\frac{\sin(\pi\alpha)}{\pi\alpha}\right)^2\! \|F^\top\|_2^2. 
\end{aligned}
\end{equation}
\noindent \textbf{Bounding $\mathbf{\sum_{j=-N}^{N} B_j\circ e^{-ij\underline{t}}c_j}$.} Let $D_{\ell,j} = \left((-1)^\ell 2\delta_j \sin (\pi\delta_j)\right)/\left(\pi(\ell^2 - \delta_j^2)\right)$. Since $|D_{\ell,j}| = \mathcal{O}(\ell^{-2})$ and $\abs{\cos(\ell t_k) e^{-ijt_k}c_j}\leq 1$, we have by the Fubini--Tonelli Theorem that
\begin{equation*}
    \sum_{j=-N}^{N} \!\!B_j \circ e^{-ij\underline{t}} c_j = \!\!\sum_{j=-N}^{N}\!\! \left[\sum_{\ell = 1}^\infty D_{\ell,j}\cos(\ell \underline{t})\right] \circ e^{-ij\underline{t}} c_j = \sum_{\ell = 1}^\infty \sum_{j=-N}^{N} \!\!\left[D_{\ell,j}\cos(\ell \underline{t}) \circ e^{-ij\underline{t}}c_j\right].
\end{equation*}
Define $\underline{v}^\ell = \sum_{j=-N}^{N} \left[D_{\ell,j}\cos(\ell \underline{t}) \circ e^{-ij\underline{t}}c_j\right]$. Since $|\cos(\ell t_k) e^{-ijt_k}c_j|\leq 1$, there exists $C > 0$, independent of $\ell$, such that $\norm{\underline{v}^\ell}_2 \leq C\max_j \abs{D_{\ell,j}}$. Since $\max_j \abs{D_{\ell,j}} = \mathcal{O}(\ell^{-2})$, we have $\sum_{\ell = 1}^\infty  \norm{\underline{v}^\ell}_2 < \infty$. Hence, by Minkowski's inequality for integrals, we have $\norm{\sum_{\ell = 1}^\infty \underline{v}^\ell}_2 \leq \sum_{\ell = 1}^\infty \norm{\underline{v}^\ell}_2$ and
\begin{align*}
    \norm{\sum_{j=-N}^{N} B_j \circ e^{-ij\underline{t}}c_j}_2 &= \norm{\sum_{\ell = 1}^\infty \underline{v}^\ell}_2 \leq \sum_{\ell = 1}^\infty \norm{\underline{v}^\ell}_2 = \sum_{\ell = 1}^\infty \norm{\cos(\ell\underline{t}) \circ \sum_{j=-N}^{N} D_{\ell,j}e^{-ij\underline{t}}c_j}_2\\
    &\stackrel{(*)}{\leq} \sum_{\ell = 1}^\infty \norm{\cos(\ell\underline{t})}_\infty \norm{\sum_{j=-N}^{N} D_{\ell,j}e^{-ij\underline{t}}c_j}_2 \\
    & \stackrel{(**)}{\leq} \sum_{\ell = 1}^\infty \max_{j}\abs{D_{\ell,j}}\norm{\sum_{j=-N}^{N} e^{-ij\underline{t}}c_j}_2 \leq \|F^\top\|_2\sum_{\ell = 1}^\infty \max_{j}\abs{D_{\ell,j}},
\end{align*}
where ($*$) follows by applying H\"older's inequality and ($**$) holds due to the orthogonality of $\{e^{-ij\underline{t}}c_j\}_{j=-N}^N$.
Since $\sum_{\ell = 1}^\infty 2\alpha/[\pi(\ell^2 - \alpha^2)]$ is the partial fraction expansion of $[1/(\pi \alpha) - \cot(\pi\alpha)]$ and $|D_{\ell,j}|$ is maximized when $|\delta_j| = \alpha$, we have
\begin{equation*}
    \sum_{\ell=1}^\infty \max_{j}\abs{D_{\ell,j}} \leq \sum_{\ell=1}^\infty \frac{2\alpha \sin(\pi\alpha)}{\pi(\ell^2-\alpha^2)} = \frac{\sin(\pi\alpha)}{\pi\alpha} - \cos(\pi\alpha).
\end{equation*}
Hence, we find that
\begin{equation}\label{eq.controlB}
    \norm{\sum_{j=-N}^{N} B_j \circ e^{-ij\underline{t}}c_j}_2 \leq \left(\frac{\sin(\pi\alpha)}{\pi\alpha} - \cos(\pi\alpha)\right)\|F^\top\|_2.
\end{equation}
\noindent \textbf{Bounding $\mathbf{\sum_{j=-N}^{N} C_j\circ e^{-ij\underline{t}}c_j}$.} Let $E_{\ell,j} \!= \!\!\left((-1)^\ell 2\delta_j\!\cos (\pi\delta_j)\right)\!/\!\!\left(\pi[(\ell-1/2)^2 \!- \delta_j^2]\right)$. Then, we have that
\begin{equation*}
    \norm{\sum_{j=-N}^{N} \!\!C_j \circ e^{-ij\underline{t}}c_j}_2\!\! \leq \!\sum_{\ell = 1}^\infty \norm{\sin\!\left(\!\left(\ell-\frac{1}{2}\right)\!\underline{t}\right) \!\circ\!\!\! \sum_{j=-N}^{N} \!\!E_{\ell,j}e^{-ij\underline{t}}c_j}_2 \!\!\!\leq \!\|F^\top\|_2\!\sum_{\ell = 1}^\infty \!\max_{j}\abs{E_{\ell,j}}.
\end{equation*}
Since $\sum_{\ell = 1}^\infty 2\alpha/(\pi[(\ell-1/2)^2-\alpha^2])$ is the partial fraction expansion of $\tan(\pi\alpha)$ and $|E_{\ell,j}|$ is maximized when $|\delta_j| = \alpha$, we find that
\begin{equation*}
    \sum_{\ell = 1}^\infty \max_{j}\abs{E_{\ell,j}} \leq \sum_{\ell = 1}^\infty \frac{2\alpha\cos(\pi\alpha)}{\pi((\ell-1/2)^2 - \alpha^2)} = \sin(\pi\alpha).
\end{equation*}
Hence, we conclude that
\begin{equation}\label{eq.controlC}
    \norm{\sum_{j=-N}^{N} C_j \circ e^{-ij\underline{t}}c_j}_2 \leq \sin(\pi\alpha) \|F^\top\|_2.
\end{equation}
The statement of the theorem follows by combining \cref{eq.controlA}, \cref{eq.controlB}, and \cref{eq.controlC}.
\end{proof}
Since $F/\sqrt{2N+1}$ is a unitary matrix, we find that $\|F\|_2 = \sqrt{2N+1}$. Therefore,~\cref{thm.kadec} shows that $\|\tilde{F}-F\|_2\leq \varphi_\alpha \sqrt{2N+1}$. It is not immediately obvious that~\cref{thm.kadec} is a discrete analogue of Kadec's theorem. We can write $\tilde{F} = F + E$ with $\|E\|_2\leq \varphi_\alpha \|F\|_2$, so we have 
\begin{equation}
    \|\tilde{F}\|_2 \leq  (1+\varphi_\alpha)\|F\|_2.
 \label{eq:NUDFTnorm}
 \end{equation} 
Moreover, since all the singular values of $F$ are $\sqrt{2N+1}$, for $0\leq \alpha<1/4$ (as this ensures that $\varphi_\alpha<1$) we find the following bound on the inverse NUDFT by Weyl's inequality:
\begin{equation} 
\|\tilde{F}^{-1}\|_2 \leq \frac{1}{\norm{F}_2 - \norm{E}_2} \leq \frac{1}{(1-\varphi_\alpha)\|F\|_2}.
\label{eq:NUDFTinversenorm}
\end{equation} 
Hence,~\cref{thm.kadec} gives the following discrete version of Kadec's theorem, which is an analogue of the frame bound for Riesz bases in sampling theory. 
\begin{corollary} 
Under the same assumptions as~\cref{thm.kadec}, we have
\begin{equation} 
(1-\varphi_\alpha)^2\|\underline{c}\|_2^2 \leq \frac{1}{2N+1}\|\tilde{F}\underline{c}\|_2^2 \leq (1+\varphi_\alpha)^2\|\underline{c}\|_2^2, \qquad \underline{c}\in\mathbb{C}^{2N+1},
\label{eq:discretekadec} 
\end{equation} 
where $\varphi_\alpha$ is given in~\cref{eq.Kadec}. 
\end{corollary} 
\begin{proof} 
The inequalities follow immediately from~\cref{eq:NUDFTnorm},~\cref{eq:NUDFTinversenorm}, and the definition of the spectral norm.
\end{proof} 
When $\alpha\geq1/4$, we believe that there is no constant $C_1>0$ that is independent of $N$ such that $C_1\|\underline{c}\|_2^2\leq \tfrac{1}{2N+1}\|\tilde{F}\underline{c}\|_2^2$. This is because Levinson has showed that $\left\{e^{i\lambda_k x}\right\}$ for $k\in\mathbb{Z}$ does not always form a Riesz basis when $|\lambda_k-k|\geq 1/4$~\cite{kadec}.  Instead, since $\tilde{F}^{-1}$ exists for any perturbed nodes with $1/4\leq \alpha < 1/2$ (as $\tilde{F}$ is a Vandermonde matrix), there must be a $C_1>0$ that depends on $N$ such that $C_1\rightarrow 0$ as $N\rightarrow \infty$. Therefore, $\alpha = 1/4$ is less of a significant threshold in the discrete setting than in sampling theory. 
From~\cite[Chapt.~4]{austin}, it is likely that one can show that $C_1 = \Omega(N^{-4\alpha})$;\footnote{For two functions $g_1(N)$ and $g_2(N)$, one writes $g_1(N) = \Omega(g_2(N))$ if there is a constant $C>0$ that is independent of $N$ such that $g_1(N)\geq Cg_2(N)$ for all $N$.} however, this is not asymptotically tight. We do not know how to improve the lower bounds on $C_1$ for $1/4\leq \alpha<1/2$. 

\subsection{The condition number of a NUDFT matrix}\label{sec:NUDFT} 
Of course,~\cref{eq:discretekadec} also has something to say about the condition number of the NUDFT matrix in~\cref{eq:NUDFTmatrix}. While there are many different versions of a NUDFT matrix, we expect that similar bounds can be derived for their condition numbers. 
\begin{corollary}
\label{cor:NUDFTconditioning}
Under the same assumptions as~\cref{thm.kadec}, the condition number of $\tilde{F}$ can be bounded independently of $N$:
\begin{equation*}
\kappa_2(\tilde{F}) = \norm{\tilde{F}}_2\norm{\tilde{F}^{-1}}_2 \leq \frac{1+\varphi_\alpha}{1-\varphi_\alpha},
\end{equation*}
where $\varphi_\alpha$ is given in~\cref{eq.Kadec}. 
\end{corollary}
\begin{proof} 
The bound follows immediately from~\cref{eq:NUDFTnorm} and~\cref{eq:NUDFTinversenorm}.
\end{proof} 
Many fast algorithms for computing $\tilde{F}^{-1}\underline{f}$ rely on a Krylov solver and an $\mathcal{O}(N\log N)$ complexity matrix-vector product for $\smash{\tilde{F}}$ and $\smash{\tilde{F}^\top}$~\cite{dutt1993fast}. The number of required iterations of the Krylov solver depends on $\kappa_2(\tilde{F})$. \Cref{cor:NUDFTconditioning} shows that for perturbed equally spaced grids with a fixed constant $0\leq \alpha<1/4$, the number of iterations is independent of $N$. This means that these Krylov-based algorithms for computing $\tilde{F}^{-1}\underline{f}$ only require $\mathcal{O}(N\log N)$ operations when $0\leq \alpha<1/4$, which theoretically justifies previous numerical observations~\cite[Fig.~4]{ruiz2018nonuniform}. In Chebfun~\cite{driscoll2014chebfun}, a Krylov-based inverse NUDFT is implemented in the \texttt{chebfun.inufft} command~\cite{ruiz2018nonuniform}: 
\begin{verbatim} 
 N = 1e4; alpha = 0.1; h = 2*pi/(2*N+1);
 tilde_x = (-N:N)'*h + alpha*(2*rand(2*N+1,1)-1)*h; % perturbed nodes
 fx = randn(2*N+1,1);                               % func samples
 cfs = chebfun.inufft(exp(-1i*tilde_x*N).*fx,tilde_x/(2*pi));%trig cfs
\end{verbatim} 
However, we believe that Krylov-based algorithms do not have an $\mathcal{O}(N\log N)$ complexity when $1/4\leq \alpha <1/2$.  Instead, in the regime $1/4\leq \alpha<1/2$, other $\mathcal{O}(N\log N)$ algorithms can be used for the inverse NUDFT~\cite[Sect.~4]{dutt1995fast}, though we are not aware of any publicly available code. 

\subsection{Marcinkiewicz--Zygmund inequalities}\label{sec:MZ}
\Cref{eq:discretekadec} can also be rephrased as an MZ inequality at perturbed nodes. While more general MZ inequalities are available in the literature (see~\cite{marzoseip}) and will be discussed in~\cref{sec:MZbounded}, the inequalities here have explicit bounds.

\begin{corollary}\label{cor.MZkadec}
Let $\tilde{x}_{-N},\ldots,\tilde{x}_N$ be $\alpha$-perturbed nodes with $0\leq\alpha<1/4$. The following MZ inequalities hold: 
\begin{equation}
\label{eq:MZinequalities}
	\frac{(1\!-\!\varphi_\alpha)^2}{2\pi} \int_{-\pi}^\pi \! |q(x)|^2 dx \leq \sum_{j=-N}^{N} \! \frac{\abs{q(\tilde{x}_{j})}^2}{2N\!+\!1} \leq \frac{(1\!+\!\varphi_\alpha)^2}{2\pi} \int_{-\pi}^\pi \! |q(x)|^2 dx, \quad q \!\in\! \mathcal{T}_N,
\end{equation}
where $\varphi_\alpha$ is given in~\cref{eq.Kadec}. 
\end{corollary}
\begin{proof}
Let $q(x) = \sum_{k=-N}^N c_k e^{ikx}=\sum_{k=-N}^N c_{-k} e^{-ikx}$ be a trigonometric polynomial in $\mathcal{T}_N$. Note that we have the matrix equation $\tilde{F}\underline{c} = \underline{q}$ with $\tilde{F}$ in~\cref{eq:NUDFTmatrix}, where $\underline{c} = (c_N, \ldots, c_{-N})^\top$ and $\underline{q} = (q(\tilde{x}_{-N}), \ldots, q(\tilde{x}_N))^\top$. The result follows from~\cref{eq:discretekadec}, the definition of $\|\tilde{q}\|_2^2$, and Parseval's Theorem, i.e., $\norm{\underline{c}}_2^2 = \tfrac{1}{2\pi}\int_{-\pi}^\pi |q(x)|^2dx$. 
\end{proof}

 The MZ inequalities in~\cref{eq:MZinequalities} confirms that the 2-norm Lebesgue constant, which is denoted by $\tilde{\Lambda}_{2N+1}^{(2)}$ (see~\cite[sect.~3.4.1]{austin}), is bounded in $N$ when $0 \leq \alpha < 1/4$. In particular, we have $\tilde{\Lambda}_{2N+1}^{(2)} \leq 1/(1-\varphi_\alpha)$ for all $0 \leq \alpha < 1/4$ and integer $N$. The MZ inequalities are also useful for studying quadrature weights and the rate of convergence of quadrature rules and interpolation at perturbed nodes~\cite{grochenig,chui,konstantin}. By calculating the explicit constants in these bounds, we can give explicit error estimates, as opposed to asymptotic convergence bounds (see~\cref{sec:interpolation,sec:quad}). 

\section{Interpolation at unevenly spaced samples}\label{sec:interpolation}
At equally spaced nodes, trigonometric interpolation enjoys rapid convergence to functions. Suppose that $f:[-\pi,\pi)\rightarrow \mathbb{C}$ is a continuous periodic function and that $q_N(x)$ is its trigonometric interpolant in $\mathcal{T}_N$ at the equally spaced nodes $x_j = jh$ for $-N\leq j\leq N$, where $h = 2\pi/(2N+1)$. Then, if $f$ is $\sigma\geq 1$ times differentiable and $f^{(\sigma)}$ is of bounded variation $V$ on $[-\pi,\pi)$, then~\cite[Thm.~4.2]{wright2015extension}
\begin{equation} 
\|f - q_N\|_{L^\infty} \leq \frac{2V}{\pi \sigma N^{\sigma}},
\label{eq:diffFunctions}
\end{equation} 
where throughout this paper, $\norm{\cdot}_{L^r} = \norm{\cdot}_{L^r([-\pi,\pi))}$ is the standard $L^r([-\pi,\pi))$ norm induced by the Lebesgue measure on $[-\pi,\pi)$. While if $f$ is analytic with $|f(z)| \leq M$ in the open strip of half-width $\rho_0>0$ along the real axis in the complex plane, one has exponential convergence~\cite[Thm.~4.2]{wright2015extension}:
\begin{equation} 
\|f - q_N\|_{L^\infty} \leq \frac{4Me^{-\rho_0 N}}{e^{\rho_0}-1}.
\label{eq:analyticFunctions} 
\end{equation} 
One wonders if trigonometric interpolants at perturbed equally spaced nodes enjoy the same kind of rapid convergence. When $\alpha<1/4$, we can show that one loses at most a factor of $\sqrt{N}$ in the convergence rate via the MZ inequality in~\cref{eq:MZinequalities}.

\begin{theorem}\label{thm.interpolatemax}
Let $f: [-\pi,\pi) \rightarrow \C$ be a continuous periodic function and $0 \leq \alpha < 1/4$. For any $\alpha$-perturbed nodes of size $2N+1$, if $\tilde{q}_N \in \mathcal{T}_N$ is the corresponding interpolant of $f$, then 
\[
\norm{f-\tilde{q}_N}_{L^\infty} \leq \left(1 + \frac{\sqrt{2N+1}}{\cos(\pi\alpha) - \sin(\pi\alpha)}\right) \min_{q\in\mathcal{T}_N} \norm{f-q}_{L^\infty}.
\]
\end{theorem}
\begin{proof}
First, note that for any trigonometric polynomial $q \in \mathcal{T}_N$ we have by Young's inequality and Parseval's Theorem that 
\[
	\norm{q}_{L^\infty} = \frac{1}{2\pi}\norm{p * q}_{L^\infty} \leq \frac{1}{2\pi} \norm{p}_{L^2} \norm{q}_{L^2} = \sqrt{\frac{2N+1}{2\pi}} \norm{q}_{L^2},
\]
where $(p\ast q)(x) = \int_{-\pi}^\pi p(x-s)q(s)ds$ and $p(x) = \sum_{j=-N}^N e^{ijx}$. If $\tilde{x}_{-N},\ldots,\tilde{x}_N$ are the perturbed nodes and $q$ is any polynomial in $\mathcal{T}_N$, then, by~\cref{eq:MZinequalities}, we have
\begin{align*}
	\norm{f-\tilde{q}_N}_{L^\infty} &\leq \norm{f-q}_{L^\infty} + \norm{q-\tilde{q}_N}_{L^\infty} \\
	&\leq \norm{f-q}_{L^\infty} + \sqrt{\frac{2N+1}{2\pi}} \norm{q-\tilde{q}_N}_{L^2} \\
	&\leq \norm{f-q}_{L^\infty} + \frac{1}{1-\varphi_\alpha} \sqrt{\sum_{j=-N}^N \abs{q(\tilde{x}_j) - f(\tilde{x}_j)}^2} \\
	&\leq \left(1 + \frac{1}{1-\varphi_\alpha} \sqrt{2N+1}\right) \norm{f-q}_{L^\infty},
\end{align*}
where $\varphi_\alpha$ is given by~\cref{eq.Kadec}.  The result follows by selecting $q$ to be the best fit polynomial in $\mathcal{T}_N$ to $f$ in $\|\cdot\|_{L^\infty}$ and noting that $1-\varphi_\alpha = \cos(\pi\alpha) - \sin(\pi\alpha)$. 
\end{proof}

Using~\cref{eq:diffFunctions} and~\cref{eq:analyticFunctions} along with the same smoothness conditions, we find that for any $0\leq \alpha<1/4$ and $\rho <\rho_0$, we have
\begin{equation} 
\|f - \tilde{q}_N\|_{L^\infty} =\mathcal{O}(N^{1/2-\sigma}), \quad \|f - \tilde{q}_N\|_{L^\infty} = \mathcal{O}(e^{-\rho N}),
\label{eq:FirstConvergenceInterpolation} 
\end{equation} 
respectively. This means that with unevenly spaced samples, one can safely use $\tilde{q}_N$ as a surrogate for $f$, and the price to pay for perturbed samples is minimal when $\alpha$ is small. In~\cref{thm.interpolateoptimal}, we show that the convergence rates in~\cref{eq:FirstConvergenceInterpolation} can be slightly improved.  Convergence results of interpolants at $\alpha$-perturbed nodes in the $L^2$ norm are also possible to derive from the MZ inequality in~\cref{eq:MZinequalities} by using~\cite[Thm~2.2]{grochenig}. 

One can easily do unevenly spaced trigonometric interpolation in Chebfun via the \texttt{chebfun.interp1} command. For example, below we compute a trigonometric polynomial that interpolates $\cos(x)$ at $\alpha$-perturbed nodes: 
\begin{verbatim} 
 N = 1e2; alpha = 0.1; h = 2*pi/(2*N+1);
 tilde_x = (-N:N)'*h + alpha*(2*rand(2*N+1,1)-1)*h; % perturbed nodes
 qN = chebfun.interp1(tilde_x, cos(tilde_x), 'periodic', [-pi pi]); 
\end{verbatim} 
Instead of using an optimal complexity algorithm based on the NUDFT, this code uses the so-called trigonometric barycentric formula~\cite{wright2015extension}. This approach has some advantages as there is a convenient way to update the formula when a new interpolation node becomes available.

\section{Quadrature at unevenly spaced samples}\label{sec:quad} 
Another important computational task is approximating integrals from knowledge of function samples at unevenly spaced nodes. While the quadrature rule at equally spaced nodes is the trapezoidal rule that enjoys perfect stability, positive quadrature weights, and rapid convergence, we wonder how much perturbing the nodes changes things. We are particularly interested in the signs of the quadrature weights~\cite{filbir,mhaskar}, the boundedness of the absolute sum of the quadrature weights~\cite{austin,mhaskar}, and the convergence rate of quadrature rules at perturbed nodes~\cite{austintrefethen,trefethenweide,grochenig}. 

\subsection{Are quadrature weights at perturbed nodes nonnegative?}
It is highly desirable to use quadrature rules for which all the weights are non-negative as such rules are perfectly stable. The absolute condition number of the integral $\int_{-\pi}^\pi f(x) dx$ is $2\pi$, while the absolute condition number of the quadrature rule in~\cref{eq:quadratureRule} is $\sum_{j=-N}^N |\tilde{w}_j|$. Since~\cref{eq:quadratureRule} is exact, we know that $\smash{\sum_{j=-N}^N \tilde{w}_j = \int_{-\pi}^\pi 1 dx = 2\pi}$ and hence the quadrature rule is perfectly stable when $\tilde{w}_j\geq 0$ for $-N\leq j\leq N$. Since the trapezoidal rule has all nonnegative weights, one might expect that quadratures at $\alpha$-perturbed nodes also have non-negative weights for sufficiently small $\alpha>0$. Surprisingly, we find that this is not case and now give a quadrature rule at $\alpha$-perturbed nodes with a negative weight for any fixed $\alpha>0$. 

For any $\alpha>0$ and integer $N$, consider the following perturbed grid where the nodes are maximally perturbed in an alternating fashion (see~\cref{fig:NegativeWeights}): 
\begin{equation} 
\tilde{x}_j = 
\begin{cases} 
(j-\alpha)h, & -N\leq j \leq -1, j=\text{ even},\\ 
(j+\alpha)h, & -N\leq j \leq -1, j=\text{ odd},\\ 
0, & j = 0,\\
(j-\alpha)h, & 1\leq j \leq N, j=\text{ odd},\\ 
(j+\alpha)h, & 1\leq j \leq N, j=\text{ even},\\ 
\end{cases} 
\qquad h = \frac{2\pi}{2N+1}.
\label{eq:perturbedQuadrature} 
\end{equation} 
\begin{figure}
\label{fig:NegativeWeights}
\centering 
\begin{tikzpicture} 
\draw[black,thick] (0,0)--(9,0); 
\draw[black,thick] (0,-.2)--(0,.2); 
\filldraw (.5,0) circle (2pt);
\filldraw (1.5,0) circle (2pt);
\filldraw (2.5,0) circle (2pt);
\filldraw (3.5,0) circle (2pt);
\filldraw[red] (4.5,0) circle (2pt);
\filldraw (5.5,0) circle (2pt);
\filldraw (6.5,0) circle (2pt);
\filldraw (7.5,0) circle (2pt);
\filldraw (8.5,0) circle (2pt);
\draw[black,thick] (9,-.25) arc (-15:15:1);
\draw [stealth-, thick] (.5-.333,.2) -- (.5,.2);
\draw [stealth-, thick] (1.5+.333,.2) -- (1.5,.2);
\draw [stealth-, thick] (2.5-.333,.2) -- (2.5,.2);
\draw [stealth-, thick] (3.5+.333,.2) -- (3.5,.2);
\draw [stealth-, thick] (5.5-.333,.2) -- (5.5,.2);
\draw [stealth-, thick] (6.5+.333,.2) -- (6.5,.2);
\draw [stealth-, thick] (7.5-.333,.2) -- (7.5,.2);
\draw [stealth-, thick] (8.5+.333,.2) -- (8.5,.2);
\filldraw[red] (.5-.333,0) circle (2pt);
\filldraw[red] (1.5+.333,0) circle (2pt);
\filldraw[red] (2.5-.333,0) circle (2pt);
\filldraw[red] (3.5+.333,0) circle (2pt);
\filldraw[red] (5.5-.333,0) circle (2pt);
\filldraw[red] (6.5+.333,0) circle (2pt);
\filldraw[red] (7.5-.333,0) circle (2pt);
\filldraw[red] (8.5+.333,0) circle (2pt);
\node at (0,-.5) {$-\pi$};
\node at (9,-.5) {$\pi$};
\node at (4.5,-.3) {$0$};
\node at (5.5,-.3) {$h$};
\node at (6.5,-.3) {$2h$};
\node at (7.5,-.3) {$3h$};
\node at (8.5,-.3) {$4h$};
\node at (3.5,-.3) {$-h$};
\node at (2.5,-.3) {$-2h$};
\node at (1.5,-.3) {$-3h$};
\node at (.5,-.3) {$-4h$};
\node at (4.5,.5) {$h=\frac{2\pi}{9}$};
\end{tikzpicture} 
\caption{For any $\alpha>0$ and sufficiently large $N$, a quadrature rule at $\alpha$-perturbed nodes with negative weights can be constructed by maximally perturbing equally spaced nodes in an alternating fashion, as depicted.}
\end{figure}
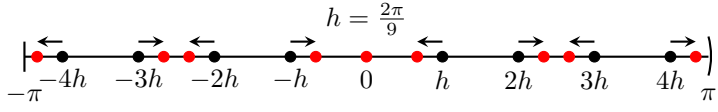
Associated with the perturbed nodes in~\cref{eq:perturbedQuadrature} is an exact quadrature rule such that 
\begin{equation} 
\int_{-\pi}^\pi q(x) dx= \sum_{j=-N}^N \tilde{w}_j q(\tilde{x}_j), \qquad q\in\mathcal{T}_N. 
\label{eq:perturbedQuadratureRule} 
\end{equation} 
We now show that $\tilde{w}_0$ is negative if $N$ is taken to be sufficiently large. This means that the quadrature rule in~\cref{eq:perturbedQuadratureRule} is unfortunately not perfectly stable. 
\begin{theorem}\label{thm.negative}
For any $\alpha > 0$ and sufficiently large even integer $N$, the quadrature weight $\tilde{w}_0$ in~\cref{eq:perturbedQuadratureRule} associated with the perturbed nodes in~\cref{eq:perturbedQuadrature} is negative. 
\end{theorem}
%
%
%
\begin{proof}
Let $\ell_0$ be the trigonometric Lagrange polynomial for $\tilde{x}_0$ associated with the perturbed nodes in~\cref{eq:perturbedQuadrature}, i.e.,~\cite{henrici}
\begin{equation} 
\ell_0(x) = \prod_{j=-N,j\neq 0}^{N}\frac{\sin\left(\frac{x-\tilde{x}_j}{2}\right)}{\sin\left(\frac{\tilde{x}_0-\tilde{x}_j}{2}\right)}.
\label{eq:lagrange}
\end{equation}
Note that $\ell_0\in\mathcal{T}_N$ satisfies $\ell_0(\tilde{x}_j) = 0$ for $-N\leq j\leq N$ and $j\neq 0$ as well as $\ell_0(\tilde{x}_0) = 1$. Therefore, using the fact that~\cref{eq:perturbedQuadratureRule} is an exact quadrature rule, we have 
\[
    \tilde{w}_{0} = \sum_{j=-N}^N \tilde{w}_j \ell_0(\tilde{x}_j) = \int_{-\pi}^\pi \ell_0(x) dx = \frac{2\pi}{2N+1} \sum_{j=-N}^{N} \ell_0( jh ),
\]
where in the last equality we used the fact that the trapezoidal rule at equally spaced nodes is exact. Now, we have
\begin{equation*}
    \tilde{w}_{0} = \frac{2\pi}{2N+1}\left(1 + \sum_{\substack{j=-N\\j \neq 0}}^{N} \ell_0(jh)\right). \\
\end{equation*} 
Since $N$ is an even integer, by~\cref{lem.Ly} and the fact that $(2/\pi)x\leq \sin(x)\leq x$ for $x \in [0,\pi/2]$, we have
\begin{equation*}
\begin{aligned}
 \sum_{\substack{j=-N\\j \neq 0}}^{N} \ell_0(jh) \leq -2\sum_{j=1}^N \frac{\sin\left(\alpha h/2\right)}{\sin\left((jh + \alpha h)/2\right)} \leq -\frac{4}{\pi} \sum_{j=1}^N \frac{\alpha h}{jh + \alpha h} \rightarrow -\infty
\end{aligned}
\end{equation*}
as $N \rightarrow \infty$. Hence, for a sufficiently large even $N$, we have $\tilde{w}_0 < 0$.
\end{proof}

\cref{thm.negative} tells us that, no matter how small, a perturbation of equally spaced nodes can cause the corresponding quadrature rule to have a negative weight. However, when $\alpha$ is small, $N$ must be extremely large for this to happen for the perturbed nodes in~\cref{eq:perturbedQuadrature}. In fact, the proof in~\cref{thm.negative} requires that $N$ is even and so large that 
\begin{equation*}
	1 -\frac{4\alpha}{\pi} \sum_{j=1}^N\frac{1}{j+\alpha} < 0,
\end{equation*}
which is ensured when $N\geq \exp({\tfrac{\pi}{4\alpha}+\tfrac{1}{2}})$.\footnote{Note that $\sum_{j =1}^{N}1/(j+\alpha) > \sum_{j=1}^N 1/(j+1) >\log(N) - 1/2$.} Thus, while quadrature at $\alpha$-perturbed nodes can have negative weights when $\alpha=1/5$ and $N$ is an even integer $\geq 83$, for $\alpha = 10^{-2}$ negative weights are only ensured when $N$ is even and $N>2.12\times 10^{34}$! One may wonder if there are more devilish perturbations than in~\cref{eq:perturbedQuadrature}, which guarantee negative weights for much smaller $N$. In~\cref{section:lowerbound}, we show that there are no such perturbations in the sense that it is required that $\log(N) = \Omega(1/\alpha)$ to have a quadrature rule at $\alpha$-perturbed nodes with a negative weight.  The analysis in~\cref{section:lowerbound} is quite explicit. For example, we know that when $\alpha=10^{-2}$, all quadrature rules at $\alpha$-perturbed nodes with $N<9.4\times 10^6$ have positive weights. So, when $\alpha$ is small, such quadrature rules are perfectly stable for all practical $N$; however, for large perturbations one does begin to get concerned about their potential numerical instability. 

\subsection{Bounding the absolute sum of the quadrature weights}\label{sec:BoundedAbsoluteSum} 
To analyze the stability of quadrature rules at $\alpha$-perturbed nodes further, we take a look at the absolute sum of the quadrature weights as this is the absolute condition number of computing the quadrature rule. Since we are considering exact quadrature rules, there is a close connection between the NUDFT matrix, $\tilde{F}$, in~\cref{eq:NUDFTmatrix} and quadrature weights. In fact, since $e^{ijx}\in\mathcal{T}_N$ for $-N\leq j\leq N$, the quadrature rule must integrate them exactly. This means that the following linear system must be satisfied:
\begin{equation} 
\tilde{F}^\top \underline{\tilde{w}} = 2\pi \underline{e}_0, \qquad \underline{e}_0 = \left(0,\ldots,0,1,0,\ldots,0\right)^\top
\label{eq:linearSystemForWeights} 
\end{equation} 
where $\underline{\tilde{w}} = (\tilde{w}_{-N},\ldots,\tilde{w}_{N})^\top$ is the vector of quadrature weights. This means that the quadrature weights associated with $\alpha$-perturbed nodes are simple to compute: 
\begin{verbatim} 
  N = 1e3; alpha = 0.1; h = 2*pi/(2*N+1); 
  tilde_x = (-N:N)'*h + alpha*(2*rand(2*N+1,1)-1)*h; % perturbed nodes
  tF = exp(-1i*tilde_x*(-N:N));                      % NUDFT matrix
  e0 = zeros(2*N+1,1); e0(N+1)=1; w = tF'\(2*pi*e0); % quad wts
\end{verbatim} 
There are also $\mathcal{O}(N\log N)$ algorithms for solving the linear system in~\cref{eq:linearSystemForWeights} as it is equivalent to an inverse NUDFT (type I)~\cite{ruiz2018nonuniform}. 

We can also use the results in~\cref{sec:kadec} to get a bound on $\sum_{j=-N}^N|\tilde{w}_j|$.
\begin{theorem}\label{thm.bounded}
Let $0 \leq \alpha < 1/4$ and $N$ be a positive integer. For any quadrature rule at $\alpha$-perturbed nodes that is exact for $\mathcal{T}_N$, we have
\begin{equation*}
    \sum_{j=-N}^{N} \abs{\tilde{w}_j} \leq \frac{2\pi}{\cos(\pi\alpha) - \sin(\pi\alpha)},
\end{equation*}
where $\tilde{w}_{-N},\ldots,\tilde{w}_N$ are the quadrature rule's weights.
\end{theorem}
\begin{proof}
Let $q(x) = \sum_{k=-N}^N c_k e^{ikx}$ be the trigonometric polynomial in $\mathcal{T}_N$ such that
\[
    q(\tilde{x}_{j}) = \text{sgn}(\tilde{w}_j), \qquad -N \leq j \leq N,
\]
where $\text{sgn}(w) = -1$ if $w < 0$ and $\text{sgn}(w) = 1$ for $w\geq 0$. Then, by~\cref{cor.MZkadec} and H\"{o}lder's inequality, we have
\begin{align*}
    \sum_{j=-N}^{N} \abs{\tilde{w}_j} &= \sum_{j=-N}^{N} \tilde{w}_jq(\tilde{x}_j) = \int_{-\pi}^\pi q(x) dx \leq \int_{-\pi}^\pi |q(x)| dx \leq \sqrt{2\pi}\left(\int_{-\pi}^\pi |q(x)|^2 dx\right)^{1/2} \\
    &\leq \frac{\sqrt{2\pi}}{\cos(\pi\alpha) - \sin(\pi\alpha)} \sqrt{\frac{2\pi}{2N+1}\sum_{j=-N}^{N}\abs{q(\tilde{x}_{j})}^2} = \frac{2\pi}{\cos(\pi\alpha) - \sin(\pi\alpha)}.
\end{align*}
\end{proof}
While quadrature rules at $\alpha$-perturbed nodes can have negative weights for any $\alpha>0$,~\cref{thm.bounded} shows that they are relatively stable for small $\alpha$. For example, when $\alpha = 1/5$, while there is a negative weight for $N\geq 83$ with the nodes in~\cref{eq:perturbedQuadrature}, the condition number of the quadrature rule is $\leq 4.53\times (2\pi)$.


\subsection{Convergence rates of quadrature rules at perturbed nodes}\label{sec:QuadratureConvergence} 
By P\'{o}lya's celebrated theorem~\cite{polya},~\cref{thm.bounded} tells us that the quadrature rule associated with perturbed nodes when $0\leq \alpha <1/4$ converges when the integrand is continuous.  That is, for any continuous periodic function $f:[-\pi,\pi)\rightarrow \mathbb{C}$ we have
 \[
\tilde{I}_N = \sum_{j=-N}^N \tilde{w}_j f( \tilde{x}_j) \rightarrow I = \int_{-\pi}^\pi f(x)dx \quad \text{ as } \quad N\rightarrow\infty
 \]
provided that $|\tilde{x}_j - jh| < \alpha h$ for $-N\leq j\leq N$ and $0\leq \alpha<1/4$, where $h = 2\pi/(2N+1)$. In fact, together with~\cite[Thm.~2.7]{grochenig} and~\cref{cor.MZkadec}, we have that 
\[
\abs{\tilde{I}_N(f) - I(f)} \leq  \frac{4\pi}{\cos(\pi\alpha) - \sin(\pi\alpha)} \min_{q\in \mathcal{T}_N} \norm{f-q}_{L^\infty}
\]
for any continuous periodic function $f$. Therefore, if $f$ is $\sigma\geq 1$ times differentiable and $f^{(\sigma)}$ is of bounded variation $V$ on $[-\pi,\pi)$, then provided $0\leq \alpha<1/4$ (see~\cref{eq:diffFunctions}),
\begin{equation} 
\abs{\tilde{I}_N(f) - I(f)} \leq \frac{8V\pi}{(\cos(\pi\alpha) - \sin(\pi\alpha))\sigma} N^{-\sigma}.
\label{eq:QuadDiffFunctions} 
\end{equation} 
Alternatively, if $f$ is analytic with $|f(z)|\leq M$ in the open strip of half-width $\rho_0>0$ along the real axis in the complex plane, we have (see~\cref{eq:analyticFunctions}) 
\[
\abs{\tilde{I}_N(f) - I(f)} \leq \frac{16 M \pi}{\cos(\pi\alpha) - \sin(\pi\alpha)} \frac{e^{-\rho_0N}}{e^{\rho_0}-1}.
\]
Theorem~1.1 of~\cite{austintrefethen} also provides convergence rates on $|I-I_N|$ for quadratures rules associated with perturbed nodes. While their convergence rates hold for any $0\leq \alpha<1/2$, our rates are a strict improvement for differentiable functions for any $0\leq \alpha<1/4$. In particular, when $\alpha$ is close to $1/4$, the convergence rate in~\cref{eq:QuadDiffFunctions} is almost an order of $N$ improvement.

\section{Further consequences of MZ inequalities}\label{sec:MZbounded}
MZ inequalities are closely connected to the convergence of interpolation and quadrature. Up to now, we have only explored MZ inequalities using the $L^2$ norm, even though MZ inequalities in other $L^r$ norms are also useful for $1\leq r\leq \infty$.  The beauty of focusing on the $L^2$-based MZ inequalities is that we derived explicit constants in the bounds (see~\cref{eq:GeneralMZ}). However, $L^r$-based MZ inequalities are extensively studied and a significant amount is known at perturbed nodes~\cite{ortega,marzoseip,lubinsky}.  

\begin{proposition}\label{prop:NoMZinequality}
Fix $1\leq r < \infty$ and constants $C_1,C_2>0$. For any $0<\alpha<1/2$ such that $\alpha\geq \min\{1/(2r), (r-1)/(2r)\}$, there exists an integer $N$ and a set of $\alpha$-perturbed nodes $\tilde{x}_{-N},\ldots,\tilde{x}_N$ such that the following inequalities do not hold: 
\begin{equation} 
C_1 \int_{-\pi}^\pi |q(x)|^r dx \leq \frac{1}{2N+1} \sum_{j=-N}^{N} \abs{q(\tilde{x}_{j})}^r \leq C_2 \int_{-\pi}^\pi |q(x)|^r dx.
\label{eq:GeneralMZ}
\end{equation} 
Conversely, if $\alpha< \min\{1/(2r), (r-1)/(2r)\}$, then there exist constants $C_1,C_2>0$ (depending on $r$) such that~\cref{eq:GeneralMZ} holds for all $\alpha$-perturbed nodes and all $q\in\mathcal{T}_N$. 
\end{proposition} 
\begin{proof} 
A proof of this statement can be found in~\cite[Thm.~1.1]{marzoseip} and~\cite[Thm.~5]{ortega}.
\end{proof}

When $r = 2$,~\cref{eq:GeneralMZ} are the MZ inequalities found in~\cref{eq:MZinequalities} and we know that one can take $C_1 = (1-\varphi_\alpha)^2/(2\pi)$ and $C_2 =  (1+\varphi_\alpha)^2/(2\pi)$.~\Cref{prop:NoMZinequality} tells us that~\cref{cor.MZkadec} is sharp in the sense that for any $\alpha>1/4$ the constants $C_1$ and $C_2$ must depend on $N$. As $r$ varies between $1\leq r<\infty$, the value of $\min\{1/(2r), (r-1)/(2r)\}$ takes the maximum value of $1/4$ at $r = 2$.  This tells us that MZ inequalities cannot help us extend our results in this paper to the $1/4\leq\alpha<1/2$ regime. \modifyadd{On the other hand, the failure of MZ inequalities when $1/4\leq\alpha<1/2$ does not seem to cause problems for interpolation and quadrature when $\alpha \geq 1/4$. In fact, for each $N$, one can find a set of $1/4$-perturbed nodes that do not satisfy an MZ inequality for any $L^r$ norm but the associated exact quadrature rule has bounded absolute sum of its weights~\cite{marzoseip}. Hence, while $\alpha = 1/4$ is a theoretically important threshold for NUDFT conditioning and MZ inequalities, we find no indication that it is a threshold for interpolation and quadrature, which agrees with the findings in~\cite{austintrefethen}.}

However, we can use~\cref{prop:NoMZinequality} to improve the rates of convergence of interpolants at perturbed nodes when $0\leq \alpha<1/4$, and derive a convergence rate that depends on $\alpha$ for differentiable functions. 
\begin{theorem}\label{thm.interpolateoptimal}
Let $f:[-\pi,\pi)\rightarrow \mathbb{C}$ be a continuous periodic function and $0 \leq \alpha < 1/4$. For any $\alpha$-perturbed nodes of size $2N+1$, let $\tilde{q}_N\in\mathcal{T}_N$ be the corresponding interpolant of $f$. Then, for any $\alpha<\alpha_0 < 1/2$, we have
\begin{equation}\label{eq.interpoptimal}
	\norm{f-\tilde{q}_N}_{L^\infty} = \mathcal{O}(N^{2{\alpha_0}}) \min_{q \in \mathcal{T}_N} \norm{f-q}_{L^\infty}.
\end{equation} 
\end{theorem}
\begin{proof}
Without loss of generality, assume $\alpha < (1-2\alpha_0)/2$.\footnote{If $\alpha \geq (1-2\alpha_0)/2$, then one can pick some $\alpha_1$ such that $\alpha < \alpha_1 < \alpha_0$ and $\alpha < (1-2\alpha_1)/2$. Since~\cref{eq.interpoptimal} holds if $\alpha_0$ is replaced by $\alpha_1$, it also holds for $\alpha_0$.} Let $r = 1/(2\alpha_0)$ so that~\cref{eq:GeneralMZ} holds and let $r' = 1/(1-2\alpha_0)$ so that $1/r + 1/r' = 1$. For any $q \in \mathcal{T}_N$, by Young's inequality, we have
\begin{equation*}
	\norm{q}_{L^\infty} = \frac{1}{2\pi}\norm{p * q}_{L^\infty} \leq \frac{1}{2\pi} \norm{p}_{L^{r'}} \norm{q}_{L^r} \leq C (2N+1)^{1/r} \norm{q}_{L^r},
\end{equation*}
where $C$ is a constant independent of $N$ and $p(x) = \sum_{j=-N}^N e^{ijx}$. Here, we also used the fact that the $L^{r'}$ norm of $p$ is $\mathcal{O}(N^{1/r})$~\cite[Lem. 2.1]{anderson}.
By~\cref{eq:GeneralMZ}, we have
\begin{align*}
	\norm{f-\tilde{q}_N}_{L^\infty} &\leq \norm{f-q}_{L^\infty} + \norm{q-\tilde{q}_N}_{L^\infty} \\
	&\leq \norm{f-q}_{L^\infty} + C (2N+1)^{1/r} \norm{q-\tilde{q}_N}_{L^r} \\
	&\leq \norm{f-q}_{L^\infty} + C C_1^{-1/r} \left(\sum_{j=-N}^N \abs{q(\tilde{x}_j) - f(\tilde{x}_j)}^r\right)^{1/r} \\
	&\leq \left(1+ C C_1^{-1/r} (2N+1)^{1/r}\right) \norm{f-q}_{L^\infty},
\end{align*}
where $C_1>0$ is independent of $N$. The result follows by selecting $q$ to be the closest polynomial in $\mathcal{T}_N$ to $f$ in $\|\cdot\|_{L^\infty}$.
\end{proof}

By controlling the Lebesgue constant, Austin and Trefethen proved that if $f$ has $\sigma$ derivatives and $\sigma > 4\alpha$, then $\norm{f-\tilde{q}_N}_{L^\infty}= \mathcal{O}(N^{4\alpha-\sigma})$ and $\abs{\tilde{I}_N(f) - I(f)} = \mathcal{O}(N^{4\alpha-\sigma})$~\cite{austintrefethen}. They also conjectured that the factor $4\alpha$ in the exponent can be improved to $2\alpha$. \Cref{thm.interpolateoptimal} proves that the convergence rate for interpolants is arbitrarily close to what they conjectured as when $f$ has $\sigma$ derivatives we know that $\min_{q \in \mathcal{T}_N} \norm{f-q}_{L^\infty} = \mathcal{O}(N^{-\sigma})$. The convergence rate for quadrature rules is even better than conjectured. (See~\cref{sec:QuadratureConvergence}.)

\section{Oversampling with unevenly spaced samples}\label{sec:oversample}
Most of the results in the paper so far assumed that $0\leq \alpha<1/4$ as MZ inequalities may not hold for all $\alpha$-perturbed nodes when $\alpha\geq 1/4$ (see~\cref{prop:NoMZinequality}). In this section, we briefly discuss oversampling, where one assumes that there are more unevenly spaced samples than the number of coefficients in the desired trigonometric approximant. A small amount of oversampling allows us to \modifyadd{use MZ inequalities to} extend our results to the $1/4\leq \alpha <1/2$ regime. The oversampling regime is studied in approximation theory~\cite{grochenig}, including the higher dimensional setting~\cite{mhaskar}. Here, we restate existing results in terms of our notation and present sharper statements in one dimension. 

Let $0\leq \alpha<1/2$ and $\varepsilon>0$ be an oversampling rate. Suppose one has $\alpha$-perturbed samples of $f$ at $\tilde{x}_{-N},\ldots,\tilde{x}_N$. Since we are oversampling, we want to find a trigonometric polynomial of degree $\leq n$ (see~\cref{eq:trigpoly}) that ``fits" the samples of the function as best as possible, where $n$ is an integer such that $n \leq \lfloor(1-\varepsilon) N\rfloor$.  Since there is typically no hope of a trigonometric polynomial interpolant, a common approach is to compute an approximant by solving a least squares problem. To do this, we first construct a tall-skinny NUDFT matrix given by 
\begin{equation} 
\tilde{F}_{jk} = e^{-i\tilde{x}_jk}, \qquad -N\leq j\leq N, \quad -n\leq k\leq n 
\label{eq:RectangularNUDFT}
\end{equation} 
and then solve $\tilde{F}\underline{c} = \underline{f}$ for the vector $\underline{c}$, where $\underline{f} = \left(f(\tilde{x}_{-N}),\ldots,f(\tilde{x}_N)\right)^\top$. The least-squares approximant to $f$ can be formed as 
\begin{equation} 
\tilde{q}_n(x) = \sum_{k=-n}^n c_{k}e^{-ikx}. 
\label{eq:LSQ}
\end{equation}
For notational consistency, we continue to assume that there is an odd number of samples; however, all results in this section hold for an even number of samples too.

\subsection{Oversampled Marcinkiewicz--Zygmund inequalities}
In~\cref{sec:kadec}, we saw the MZ inequality is closely related to the NUDFT condition number via Parseval's theorem. Below, we state the MZ inequality in the oversampling case and then discuss its consequences for the condition number of $\tilde{F}$ in~\cref{eq:RectangularNUDFT}. 

Let $0\leq \alpha <1/2$ and $\varepsilon>0$. If $\tilde{x}_{-N},\ldots,\tilde{x}_N$ are $\alpha$-perturbed nodes and $n = \lfloor(1-\varepsilon) N\rfloor$. Then, it is easy to verify that 
\begin{equation*}
	\liminf_{R \rightarrow \infty} \left(\liminf_{N \rightarrow \infty} \frac{\min_{x \in [-\pi,\pi)} \abs{\{\tilde{x}_j\}_{j=-N}^{N} \cap (x,x+R/n)}}{R} \right) \geq \frac{1}{(1-\varepsilon)\pi}.
\end{equation*}
When this condition is satisfied, Ortega-Cerd\`a and Saludes~\cite{ortega} proved that the MZ inequality in~\cref{eq:GeneralMZ} holds for any $1 \leq r < \infty$ and $q\in\mathcal{T}_N$, where the constants $C_1, C_2 > 0$ are independent of $N, n, $ and $\tilde{x}_j$, but possibly depend on $\alpha, \varepsilon$, and $r$. 

\subsection{The condition number of a rectangular NUDFT matrix}
Let $\tilde{F}$ be the NUDFT matrix given in~\cref{eq:RectangularNUDFT}. Then, for any vector $\underline{c} \in \C^{2n+1}$, we can rewrite~\cref{eq:GeneralMZ} for $r = 2$ using Parseval's theorem to obtain
\begin{equation*}
2\pi C_1 \norm{\underline{c}}_2^2 \leq \frac{1}{2N+1} \norm{\tilde{F}\underline{c}}_2^2 \leq 2\pi C_2 \norm{\underline{c}}_2^2,
\end{equation*}
where $C_1,C_2>0$ are constants that only depend on $\alpha$ and $\varepsilon$. This immediately gives us a bound on the spectral norm of $\tilde{F}$, i.e., $\|\tilde{F}\|_2 \leq \sqrt{2\pi C_2(2N+1)}$. We also have
\begin{equation*}
	\norm{\tilde{F}^\dagger \underline{f}}_2^2 \leq \frac{1}{2\pi C_1(2N+1)} \norm{\tilde{F}\tilde{F}^\dagger  \underline{f}}_2^2 \leq \frac{1}{2\pi C_1(2N+1)} \norm{\underline{f}}_2^2, \qquad \underline{f} \in \C^{2N+1},
\end{equation*}
where $\tilde{F}^\dagger$ is the Moore--Penrose pseudoinverse of $\tilde{F}$. Here, the second inequality follows from the fact that $\|\tilde{F}\tilde{F}^\dagger\|_2 = 1$. We find that $\|\tilde{F}^\dagger\|_2 \leq 1/\sqrt{2\pi C_1(2N+1)}$ and hence, the condition number of $\tilde{F}$ is bounded independently of $N$ and $n$: 
\[
	\kappa_2(\tilde{F}) = \norm{\tilde{F}}_2 \norm{\tilde{F}^\dagger}_2 \leq \sqrt{\frac{C_2}{C_1}}.
\]
We find this a remarkable bound because as soon as there is a small amount of oversampling, the condition number of $\tilde{F}$ can be bounded independently of $N$ and $n$, even in the $1/4\leq \alpha<1/2$ regime. In contrast, we believe that when $N = n$, the condition number of $\tilde{F}$ grows slowly with $N$ when $1/4\leq \alpha<1/2$.  

\subsection{Convergence of least-squares approximation at unevenly spaced samples}
Let $f:[-\pi,\pi)\rightarrow \mathbb{C}$ be a continuous periodic function and let $\tilde{q}_n$ be the least-squares approximant in~\cref{eq:LSQ}. Then, by definition of $\tilde{q}_n$, we have that $\sum_{j=-N}^N \abs{f(\tilde{x}_j) - \tilde{q}_n(\tilde{x}_j)}^2 \leq \sum_{j=-N}^N \abs{f(\tilde{x}_j) - q(\tilde{x}_j)}^2$ for any $q \in \mathcal{T}_n$. Hence, we find that 
\begin{align*}
	&\sum_{j=-N}^N \abs{q(\tilde{x}_j) - \tilde{q}_n(\tilde{x}_j)}^2 \leq \sum_{j=-N}^N (\abs{q(\tilde{x}_j) - f(\tilde{x}_j)} + \abs{f(\tilde{x}_j) - \tilde{q}_n(\tilde{x}_j)})^2 \\
	&= \sum_{j=-N}^N (\abs{q(\tilde{x}_j) - f(\tilde{x}_j)}^2 + \abs{f(\tilde{x}_j) - \tilde{q}_n(\tilde{x}_j)}^2 + 2\abs{q(\tilde{x}_j) - f(\tilde{x}_j)} \abs{f(\tilde{x}_j) - \tilde{q}_n(\tilde{x}_j)}) \\
	&\leq 4 \sum_{j=-N}^N \abs{q(\tilde{x}_j) - f(\tilde{x}_j)}^2,
\end{align*}
where the last inequality follows from the Cauchy--Schwarz inequality. By a similar argument to the proof of~\cref{thm.interpolatemax}, we find that 
\[
\left\|f - \tilde{q}_n\right\|_{L^\infty} \leq \left(1 + \sqrt{\frac{2(2n+1)}{C_1\pi}}\right) \min_{q\in\mathcal{T}_n} \|f - q\|_{L^\infty}.
\]
The reader can now use their favorite bounds on $\min_{q\in\mathcal{T}_n} \|f - q\|_{L^\infty}$ (see~\cref{eq:diffFunctions,eq:analyticFunctions}). 

\subsection{Nonexact quadrature rules}
There also exists a quadrature rule that is exact for all $q \in \mathcal{T}_n$ such that~\cite[Thm.~2.7]{grochenig}
\[
\abs{\tilde{I}_n(f) - I(f)} \leq 2\pi \left(1 + \sqrt{\frac{C_2}{C_1}}\right) \min_{q\in \mathcal{T}_n} \norm{f-q}_{L^\infty}, \quad \tilde{I}_n(f) = \sum_{j=-N}^N \tilde{w}_j f(\tilde{x}_j)
\]
for all continuous periodic $f$. In Lemma~3.6 of~\cite{grochenig}, it is shown that the quadrature weights $\tilde{w}_{j}$ are the least-squares solution of the underdetermined system $\tilde{F}^\top \underline{\tilde{w}} = 2\pi \underline{e}_0$, where $\underline{e}_0$ is the zero vector except at its central entry is $1$.  

From numerical experiments, we observer that while oversampling usually avoids negative weights and also reduces the absolute sum of the weights, it is much more unpredictable as $n$ and $N$ vary. 
We know that oversampling by a minimal amount gives us good rates of convergence for any $\alpha < 1/2$. A natural question to ask is how much we need to oversample before we can guarantee that the quadrature weights  are non-negative. Again, this problem was studied in \cite{mhaskar} for higher dimensions. A translation of their one dimensional technique states that if $n\leq N/\pi$, then for every $N$ there exists a quadrature rule, $\tilde{I}_n$, that is exact for $q \in \mathcal{T}_n$ and has non-negative weights.  


\appendix

\section{A quadrature weight at perturbed nodes is negative}\label{section.proofnegative}
To show that the perturbed nodes in~\cref{eq:perturbedQuadrature} have an associated quadrature weight that is negative when $N$ is sufficiently large, we need a few trigonometric inequalities that are technical to derive.

Let $p, q, r > 0$ be such that $0 < q-p-2r < q-p+2r < \pi$. The following trigonometric inequality holds: 
\begin{equation}\label{eq.trigtech}
\begin{aligned}
    \sin(p-r)\sin(q+r) &= \frac{\cos(q-p+2r)-\cos(p+q)}{2} \\
    &< \frac{\cos(q-p-2r)-\cos(p+q)}{2} = \sin(p+r)\sin(q-r).
\end{aligned}
\end{equation} 
The next lemma bounds the values of $\ell_0$, the trigonometric Lagrange polynomial for $\tilde{x}_0$ associated with $\{\tilde{x}_j\}_{j=-N}^N$ (see~\cref{eq:lagrange}), at an equally spaced node. This helps us bound the weight $\tilde{w}_0$ in~\cref{thm.negative}.

\begin{lemma}\label{lem.Ly}
Suppose $N$ is even and $\tilde{x}_{-N},\ldots,\tilde{x}_N$ are the perturbed nodes given in~\cref{eq:perturbedQuadrature}. We have 
\begin{equation}\label{eq.crucialA}
	\abs{\ell_0(kh)} = -\ell_0(kh) \geq \frac{\sin\!\left( \alpha h/2\right)}{\sin\!\left((\abs{k} + \alpha)h/2\right)}, \qquad -N \leq k \leq N,\quad k\neq 0, 
\end{equation}
where $h = 2\pi/(2N+1)$.
\end{lemma}
\begin{proof}
Since $\ell_0(-x) = \ell_0(x)$ for all $x$, it suffices to show that~\cref{eq.crucialA} holds for $k > 0$. By a simple counting argument, one can verify that $\ell_0(kh) = \prod_{j=-N, j \neq 0}^N (\sin((kh - \tilde{x}_j)/2) / \sin(-\tilde{x}_j/2)) \leq 0$. Let $s(x) = \sin(\abs{xh/2})$. Then, the numerator of $\abs{\ell_0(kh)}$ can be written as
\begin{align*}
&\prod_{\substack{j = -N, \\ j \neq 0}}^N \sin \left(\abs{\frac{kh-\tilde{x}_{j}}{2}}\right) \\
&= \!\! \!\!\prod_{j=-N}^{-N+k-1}  \!\!\!\!\!\sin \! \left(\abs{\frac{kh-\tilde{x}_{j}}{2}}\right) \!\! \!\!\prod_{j=-N+k}^{-1} \!\!\!\!\!\sin\! \left( \abs{\frac{kh-\tilde{x}_{j}}{2}}\right) \!\! \prod_{j=1}^{k}  \!\sin\! \left( \abs{\frac{kh-\tilde{x}_{j}}{2}}\right)  \!\!\!\prod_{j=k+1}^{N} \!\! \!\sin\! \left(\abs{\frac{kh-\tilde{x}_{j}}{2}}\right) \\
&= \!\!\!\!\prod_{j=N-k+1}^N \!\!\!\!s(j \!+\! (-1)^{j+k}\alpha) \!\!\!\!\prod_{j=k+1}^N \!\!\!\!s(j \!+\! (-1)^{j+k}\alpha) \!\!\prod_{j=0}^{k-1} \!\!s(j \!+\! (-1)^{j+k+1}\alpha) \!\!\prod_{j=1}^{N-k} \!\!s(j \!+\! (-1)^{j+k}\alpha) \\
&= \prod_{j=1}^N s(j \!+\! (-1)^{j+k}\alpha) \prod_{j=k+1}^N s(j \!+\! (-1)^{j+k}\alpha) \prod_{j=0}^{k-1} s(j \!+\! (-1)^{j+k+1}\alpha).
\end{align*}
Similarly, the denominator of $\abs{\ell_0(kh)}$ can be written as
\begin{align*}
\prod_{\substack{j = -N, \\ j \neq 0}}^N \sin \left(\abs{\frac{\tilde{x}_j}{2}}\right) = \prod_{j=1}^N \sin^2 \left(\abs{\frac{\tilde{x}_j}{2}}\right) = \prod_{j=1}^N (s(j + (-1)^j\alpha))^2.
\end{align*}
Using~\cref{eq.trigtech}, we find the for every odd $1 \leq j \leq N$, we have
\begin{equation}\label{eq.balancetrig}
s(j + \alpha)s((j+1) - \alpha) \geq s(j - \alpha)s((j+1) + \alpha).
\end{equation}
Hence, for every odd $m_1$ and even $m_2$ such that $1\leq m_1<m_2\leq N$,~\cref{eq.balancetrig} gives us
\begin{equation}\label{eq.trigprod}
\prod_{j=m_1}^{m_2} s(j + (-1)^{j+1}\alpha) \geq \prod_{j=m_1}^{m_2} s(j + (-1)^{j}\alpha).
\end{equation}
We now consider the cases when $k$ is even and odd separately. 

\noindent{\textbf{Case I: $\mathbf{k}$ is even.}} If $k$ is even, then we have
\begin{align*}
-\ell_0(kh) &= \abs{\ell_0(kh)} = \left(\prod_{j=1}^{k-2} \frac{s(j + (-1)^{j+1} \alpha)}{s(j + (-1)^{j} \alpha)}\right) \frac{s(k-1 + \alpha)}{s(k-1 - \alpha)} \frac{s(\alpha)}{s(k+\alpha)}\\
&\geq \frac{s(\alpha)}{s(k+\alpha)} = \frac{\sin(\alpha h/2)}{\sin((k+\alpha)h/2)},
\end{align*}
where we used~\cref{eq.trigprod} with $m_1 = 1$ and $m_2 = k-2$.

\noindent{\textbf{Case II: $\mathbf{k}$ is odd.}} If $k$ is odd, then we have the following inequality on $-\ell_0(kh)$: 
\begin{align*}
-\ell_0(kh) &= \abs{\ell_0(kh)} \geq \left(\prod_{j=k+2}^{N} \frac{s(j + (-1)^{j+1} \alpha)}{s(j + (-1)^{j} \alpha)}\right) \frac{s(k+1 - \alpha)}{s(k+1 + \alpha)} \frac{s(\alpha)}{s(k-\alpha)} \\
&\geq \frac{s(k+1 - \alpha)}{s(k+1 + \alpha)} \frac{s(\alpha)}{s(k-\alpha)} \geq \frac{s(\alpha)}{s(k+\alpha)} = \frac{\sin(\alpha h/2)}{\sin((k+\alpha)h/2)},
\end{align*}
where the sequence of three inequalities are obtained from~\cref{eq.trigprod} by setting (1) $m_1 = 1$ and $m_2 = N$, (2) $m_1 = k+2$ and $m_2 = N$, and (3) $m_1 = k$ and $m_2 = k+1$, respectively.  The proof is complete.
\end{proof}

\section{Only large quadrature rules at perturbed nodes can have negative weights}\label{section:lowerbound}
For any $0<\alpha<1/2$, we define $\Nmin$ to be the smallest integer such that for $N = \Nmin$ there exists a set of $\alpha$-perturbed quadrature nodes $\{\tilde{x}_j\}_{j=-N}^{N}$ so that the associated exact quadrature rule on $\mathcal{T}_N$ has a negative weight. By~\cref{thm.negative}, we know that $\Nmin$ is finite for every $\alpha$. Here, we derive an implicit lower bound~\cref{eq.L1} on $\Nmin$ and provide a closed formula in~\cref{eq.N}, which proves $\log( \Nmin ) = \Theta(\alpha^{-1})$ as $\alpha \rightarrow 0$. This is a rather technical result. 

First, we define an equivalence relation on the indices. Let $-N \leq j, k \leq N$ and define $d(j,k) = \min\{\abs{j-k},2N+1-\abs{j-k}\}$. For a fixed $j$, define an equivalence class on $\{-N, \ldots, j-1, j+1, \ldots, N\}$ by $k_1 \sim_j k_2$ if and only if $d(j,k_1) = d(j,k_2)$ and denote the equivalence class by $S^j_{d(j,k_1)}$. Note that each of $S^j_1, \ldots, S^j_N$ contains exactly $2$ elements.

\begin{lemma}\label{lem.lagresidual}
Let $0\leq \alpha < 1/2$, $h = 2\pi/(2N+1)$, and $\{\tilde{x}_k\}_{k=-N}^{N}$ be a set of $\alpha$-perturbed nodes. For fixed $-N\leq j\leq N$ and for any $-N \leq k \leq N$ with $k \neq j$, we have
\begin{equation}
    \abs{\frac{\prod_{m=-N,m \neq j,k}^{N} \sin \!\left(\frac{x_{k} - \tilde{x}_m}{2}\right)}{\prod_{m=-N,m\neq j}^{N} \sin\! \left(\frac{\tilde{x}_j - \tilde{x}_m}{2}\right)}} \leq \frac{1}{\sin\!\left(\frac{d(j,k)h}{2}\right)} \prod_{m=1}^N \frac{\sin^2\!\left(\frac{mh + \alpha h}{2}\right)}{\sin\!\left(\frac{mh - 2\alpha h}{2}\right) \sin\!\left(\frac{mh}{2}\right)},
    \label{eq:ineq1}
\end{equation}
where $x_k = \tilde{x}_j + (k-j)h$.\footnote{Note that we do not necessarily have $\abs{x_k - \tilde{x}_k} \leq \alpha h$ for all $k$. Instead, we have $\abs{x_k - \tilde{x}_k} \leq 2\alpha h$.} 
\end{lemma}
\begin{proof}
Let $d = d(j,k)$ and $j'$ be the element in $S^k_{d}$ that is not equal to $j$. Then, we can write the numerator and denominator of~\cref{eq:ineq1} as 
\begin{align*}
    A &= \prod_{\substack{m=-N,\\m \neq j,k}}^{N}\! \sin \left(\!\frac{x_{k} \!-\! \tilde{x}_m}{2}\!\right) = \sin \left(\!\frac{x_k \!-\! \tilde{x}_{j'}}{2}\!\right) \!\prod_{\substack{m=1, m \neq d, \\ S^k_m = \{k_1, k_2\}}}^{N}\! \left[\sin \left(\!\frac{x_{k} \!-\! \tilde{x}_{k_1}}{2}\!\right) \sin \left(\frac{x_{k} \!-\! \tilde{x}_{k_2}}{2}\!\right)\right],\\
    B &= \prod_{\substack{m=-N,\\m\neq j}}^{N} \! \sin \left(\!\frac{\tilde{x}_j \!-\! \tilde{x}_m}{2}\!\right) = \!\prod_{\substack{m=1,\\ S^j_m = \{k_1, k_2\}}}^{N}\! \left[\sin \left(\!\frac{\tilde{x}_j \!-\! \tilde{x}_{k_1}}{2}\!\right) \sin \left(\!\frac{\tilde{x}_j \!-\! \tilde{x}_{k_2}}{2}\!\right)\right].
\end{align*}
Suppose $S_m^k = \{k_1,k_2\}$, where $1 \leq m \leq N$. Then, we know that $d_\TT(e^{ix_k}, e^{i\tilde{x}_{k_1}}) + d_\TT(e^{ix_k}, e^{i\tilde{x}_{k_2}}) \leq 2mh + 2\alpha h$
and $d_\TT(e^{ix_k}, e^{i\tilde{x}_{k_1}}), d_\TT(e^{ix_k}, e^{i\tilde{x}_{k_2}}) \leq mh+2\alpha h$, where $d_\TT(e^{ix}, e^{iy}) = \min\{\abs{x-y}, 2\pi - \abs{x-y}\}$ is the ``distance" between $e^{ix}$ and $e^{iy}$ on the unit circle for $x,y\in[-\pi,\pi)$. Therefore, by elementary calculus, we have
\begin{equation*}
    \abs{\sin \left(\frac{x_{k} - \tilde{x}_{k_1}}{2}\right) \sin \left(\frac{x_{k} - \tilde{x}_{k_2}}{2}\right)} \leq \sin^2 \left(\frac{mh+\alpha h}{2}\right),
\end{equation*}
where the lefthand side of the inequality above is maximized when $d_\TT(e^{ix_k}, e^{i\tilde{x}_{k_1}}) = d_\TT(e^{ix_k}, e^{i\tilde{x}_{k_2}}) = mh+\alpha h$. Similarly, suppose $S_m^j = \{k_1,k_2\}$. We have $d_\TT(e^{i\tilde{x}_j}, e^{i\tilde{x}_{k_1}}) + d_\TT(e^{i\tilde{x}_j}, e^{i\tilde{x}_{k_2}}) \geq 2mh-2\alpha h$
and $d_\TT(e^{i\tilde{x}_j}, e^{i\tilde{x}_{k_1}}), d_\TT(e^{i\tilde{x}_j}, e^{i\tilde{x}_{k_2}}) \geq mh-2\alpha h$. Hence,
\begin{equation*}
    \abs{\sin \left(\frac{\tilde{x}_j - \tilde{x}_{k_1}}{2}\right) \sin \left(\frac{\tilde{x}_j - \tilde{x}_{k_2}}{2}\right)} \geq \sin \left(\frac{mh - 2\alpha h}{2}\right)\sin \left(\frac{mh}{2}\right) ,
\end{equation*}
where the lefthand side of the inequality above is minimized when $d_\TT(e^{i\tilde{x}_j}, e^{i\tilde{x}_{k_i}}) = mh$ and $d_\TT(e^{i\tilde{x}_j}, e^{i\tilde{x}_{k_{3-i}}}) = mh - 2\alpha h$ for $i = 1$ or $2$. This gives us
\begin{align*}
    \abs{\frac{A}{B}} \leq \frac{\sin\left(\!\frac{dh + 2\alpha h}{2}\!\right) \prod_{m=1, m \neq d}^{N} \sin^2\left(\!\frac{mh+\alpha h}{2}\!\right)}{\prod_{m=1}^N \sin\left(\!\frac{mh - 2\alpha h}{2}\!\right)\sin\left(\!\frac{mh}{2}\!\right) } \leq \frac{1}{\sin\left(\!\frac{dh}{2}\!\right)} \prod_{m=1}^N \frac{\sin^2\left(\!\frac{mh+\alpha h}{2}\!\right)}{\sin\left(\!\frac{mh - 2\alpha h}{2}\!\right) \sin\left(\!\frac{mh}{2}\!\right)},
\end{align*}
as desired.
\end{proof}

It is worth observing that $A/B$ in the proof of~\cref{lem.lagresidual} is almost the Lagrange polynomial at $x_j$. This connection is made precise in the next lemma.

\begin{lemma}\label{lem.controllagrange}
Using the same notation as~\cref{lem.lagresidual}, if $\ell_j$ is the $j$th trigonometric Lagrange basis polynomial for $\tilde{x}_j$ associated with $\{\tilde{x}_k\}_{k=-N}^{N}$, then 
\begin{equation}\label{eq.lem1}
    \abs{\ell_j(x_{k_1}) + \ell_j(x_{k_2})} \leq \frac{\pi\alpha}{d} \prod_{m=1}^N \frac{(m + \alpha)^2}{(m - 2\alpha)m},
    \end{equation}
where $S^j_d = \{k_1, k_2\}$.
\end{lemma}
\begin{proof}
Let $d = d(j,k)$ and suppose that $1 \leq d \leq N$. For $i = 1, 2$, we have
\begin{align*}
    \ell_j(x_{k_i}) &= \frac{\prod_{m=-N,m\neq j}^{N} \sin \left(\frac{x_{k_i} - \tilde{x}_m}{2}\right)}{\prod_{m=-N,m\neq j}^{N} \sin \left(\frac{\tilde{x}_j - \tilde{x}_m}{2}\right)} = \frac{\sin \left(\frac{x_{k_i} - \tilde{x}_{k_i}}{2}\right) \prod_{m=-N,m\neq j, k_i}^{N} \sin \left(\frac{x_{k_i} - \tilde{x}_m}{2}\right)}{\prod_{m=-N,m\neq j}^{N} \sin \left(\frac{\tilde{x}_j - \tilde{x}_m}{2}\right)}.
\end{align*}
By~\cref{lem.lagresidual}, we find that 
\begin{align*}
    \abs{\frac{\prod_{m=-N,m\neq j, k_i}^{N} \sin \left(\frac{x_{k_i} - \tilde{x}_m}{2}\right)}{\prod_{m=-N,m\neq j}^{N} \sin \left(\frac{\tilde{x}_j - \tilde{x}_m}{2}\right)}} \leq \frac{\pi}{dh} \prod_{m=1}^N \frac{(mh + \alpha h)^2}{(mh - 2\alpha h)(mh)} = \frac{\pi}{dh} \prod_{m=1}^N \frac{(m + \alpha)^2}{(m - 2\alpha)m},
\end{align*}
where we need the fact that $(2/\pi)x \leq \sin x$ for all $0 \leq x \leq \pi/2$ and the fact that $\sin y / \sin x \leq y / x$ for all $0 < x \leq y < \pi$. We now prove~\cref{eq.lem1} by considering the cases $\ell_j(x_{k_1})\ell_j(x_{k_2}) \leq 0$ and $\ell_j(x_{k_1})\ell_j(x_{k_2}) > 0$ separately. 

\noindent {\bf Case 1: $\mathbf{\ell_j(x_{k_1})\ell_j(x_{k_2}) \leq 0}$.}
Since $d_\TT(e^{ix_{k_i}}, e^{i\tilde{x}_{k_i}}) \leq 2 \alpha h$ for $i = 1, 2$, we have
\begin{align*}
    &\abs{\ell_j(x_{k_1}) + \ell_j(x_{k_2})} \leq \max\{\abs{\ell_j(x_{k_1})}, \abs{\ell_j(x_{k_2})}\} \\
    &\qquad \leq \max_{i = 1, 2}\left[\frac{d_\TT(e^{ix_{k_i}}, e^{i\tilde{x}_{k_i}})}{2}\right] \frac{\pi}{dh} \prod_{m=1}^N \frac{(m + \alpha)^2}{(m - 2\alpha)m} \leq \frac{\pi\alpha}{d} \prod_{m=1}^N \frac{(m + \alpha)^2}{(m - 2\alpha)m},
\end{align*}
where we used the fact that $\sin(x) \leq x$ for all $0 \leq x \leq \pi/2$.

\noindent {\bf Case 2: $\mathbf{\ell_j(x_{k_1})\ell_j(x_{k_2}) > 0}$.}
We claim that
\begin{equation*}
    \frac{\prod_{m=-N,m\neq j, k_1}^{N} \sin \left(\frac{x_{k_1} - \tilde{x}_m}{2}\right)}{\prod_{m=-N,m\neq j}^{N} \sin \left(\frac{\tilde{x}_j - \tilde{x}_m}{2}\right)} \frac{\prod_{m=-N,m\neq j, k_2}^{N} \sin \left(\frac{x_{k_2} - \tilde{x}_m}{2}\right)}{\prod_{m=-N,m\neq j}^{N} \sin \left(\frac{\tilde{x}_j - \tilde{x}_m}{2}\right)} < 0.
\end{equation*}
The claim follows because if we assume, without loss of generality,\footnote{Otherwise, swap the roles of $k_1$ and $k_2$.} that $k_1 < k_2$, then there are an even number of integers $m$, not including $j$, such that $k_1 < m < k_2$.\footnote{Note that $j$ may not be between $k_1$ and $k_2$. However, the claim follows regardless of how $j$ compares to $k_1$ and $k_2$.} When $m \neq k_1, k_2$, the signs of $\sin((x_{k_1}-\tilde{x}_m)/2)$ and $\sin((x_{k_2}-\tilde{x}_m)/2)$ are different if and only if $k_1 < m < k_2$. The claim follows from the fact that $\sin((x_{k_2}-\tilde{x}_{k_1})/2) \sin((x_{k_1}-\tilde{x}_{k_2})/2) < 0$ because the remaining terms multiply to a positive number. Hence, we have $\sin ((x_{k_1} - \tilde{x}_{k_1})/2) \sin ((x_{k_2} - \tilde{x}_{k_2})/2) < 0$ so that $(x_{k_1} - \tilde{x}_{k_1})(x_{k_2} - \tilde{x}_{k_2}) < 0$. Let $\delta := \tilde{x}_j - jh$. Then, $x_{k_i} - \tilde{x}_{k_i} < \alpha h + \delta$ if $x_{k_i} - \tilde{x}_{k_i} > 0$ and $\tilde{x}_{k_i} - x_{k_i} < \alpha h - \delta$ if $x_{k_i} - \tilde{x}_{k_i} < 0$. Hence, $(x_{k_1} - \tilde{x}_{k_1})(x_{k_2} - \tilde{x}_{k_2}) < 0$ implies $d_\TT(e^{ix_{k_1}}, e^{i\tilde{x}_{k_1}}) + d_\TT(e^{ix_{k_2}}, e^{i\tilde{x}_{k_2}}) \leq 2\alpha h$. By elementary calculus, we find that
\begin{equation*}
    \abs{\sin \left(\frac{x_{k_1} - \tilde{x}_{k_1}}{2}\right)} + \abs{\sin \left(\frac{x_{k_2} - \tilde{x}_{k_2}}{2}\right)} \leq 2\sin\left(\frac{\alpha h}{2}\right) \leq \alpha h.
\end{equation*}
Putting this together we obtain
\begin{align*}
    &\abs{\ell_j(x_{k_1}) + \ell_j(x_{k_2})} = \abs{\ell_j(x_{k_1})} + \abs{\ell_j(x_{k_2})} \\
    &\;\;\leq \left(\abs{\sin \left(\!\!\frac{x_{k_1} \!-\! \tilde{x}_{k_1}}{2}\!\!\right)} + \abs{\sin \left(\!\!\frac{x_{k_2} \!-\! \tilde{x}_{k_2}}{2}\!\!\right)}\right) \frac{\pi}{dh} \prod_{m=1}^N\! \frac{(m + \alpha)^2}{(m - 2\alpha)m} \leq \frac{\pi\alpha}{d} \prod_{m=1}^N\! \frac{(m + \alpha)^2}{(m - 2\alpha)m},
\end{align*}
as desired. 
\end{proof}

We are now ready to prove a lower bound on $\Nmin$. 
\begin{theorem}\label{thm.thetaN}
We have
\begin{equation}\label{eq.L1}
    g(\Nmin) > \frac{1}{\pi\alpha}, \qquad g(N) := \prod_{m=1}^{N} \left[\frac{(m + \alpha)^2}{(m - 2\alpha)m}\right] \left[\sum_{d=1}^{N} \frac{1}{d}\right]
\end{equation}
for $0<\alpha<1/2$. 
\end{theorem}
\begin{proof}
We define $g: \N \rightarrow \R$ as in~\cref{eq.L1}. It suffices to show that $g(N) \leq 1/(\pi\alpha)$ implies that a quadrature rule at $\alpha$-perturbed nodes of degree $N$ contains no negative weight. Let $\{\tilde{x}_j\}_{j=-N}^{N}$ be a set of $\alpha$-perturbed nodes. Let $\tilde{w}_j$ be the quadrature weight associated with $\tilde{x}_j$ and let $\ell_j$ be the corresponding trigonometric Lagrange basis polynomial. Let $\{x_k\}_{k=-N}^{N}$, which may depend on $j$, be defined as in~\cref{lem.lagresidual}. Then, we have $\tilde{w}_j = (2\pi/(2N+1))\sum_{k=-N}^{N} \ell_j(x_k)$. By~\cref{lem.controllagrange}, we have
\begin{align}
    &\sum_{k=-N}^{N} \ell_j(x_k) = \ell_j(x_j) + \sum_{k=-N, k \neq j}^{N} \ell_j(x_k) \geq 1 - \sum_{d=1, S^j_d=\{k_1,k_2\}}^{N} \abs{\ell_j(x_{k_1}) + \ell_j(x_{k_2})} \nonumber \\
    &\qquad \geq 1 - \sum_{d=1}^{N} \left[\frac{\pi\alpha}{d} \prod_{m=1}^N \frac{(m + \alpha)^2}{(m - 2\alpha)m}\right] = 1-\pi\alpha\prod_{m=1}^N \left[\frac{(m + \alpha)^2}{(m - 2\alpha)m}\right] \left[\sum_{d=1}^N \frac{1}{d}\right] \geq 0. \nonumber
\end{align}
This proves all quadrature weights are non-negative.
\end{proof}

For small $\alpha$, we can make the statement in~\cref{thm.thetaN} more explicit. 
\begin{corollary} 
For $0<\alpha<0.15$ and a real number $L > 0$ such that 
\begin{equation}\label{eq.L}
    (\alpha + L)e^{4L} \leq \frac{\Gamma(1+\alpha)^2}{\pi\Gamma(1-2\alpha)},
\end{equation}
where $\Gamma$ is the gamma function, we find that $\Nmin$ satisfies the following inequality: 
\begin{equation}\label{eq.N}
    \log (\Nmin+1+\alpha) + \frac{1}{2\Nmin-1} > (1-\gamma) + \frac{L}{\alpha},
\end{equation}
where $\gamma \approx 0.57722$ is the Euler--Mascheroni constant~\cite{euler}.
\label{cor:Final} 
\end{corollary} 
\begin{proof} 
We aim to show that~\cref{eq.L} implies~\cref{eq.N}. We first use Gautschi’s inequality~\cite{wendel} to obtain
\begin{equation}\label{eq.gautschi}
\prod_{m=1}^N \frac{(m \!+\! \alpha)^2}{(m \!-\! 2\alpha)m} = \frac{\Gamma(1\!-\!2\alpha) \Gamma(N\!+\!1\!+\!\alpha)^2}{\Gamma(N\!+\!1\!-\!2\alpha)\Gamma(N\!+\!1)\Gamma(1\!+\!\alpha)^2} \leq \frac{\Gamma(1\!-\!2\alpha)}{\Gamma(1\!+\!\alpha)^2} (N\!+\!1\!+\!\alpha)^{4\alpha}.
\end{equation}
Assume that $\log (N+1+\alpha) + 1/(2N-1) \leq (1-\gamma) + L/\alpha$, for an integer $N>0$. We have $\log(N+1+\alpha) \leq L/\alpha$. Hence, $N+1+\alpha \leq e^{L/\alpha}$. By~\cref{eq.gautschi}, we find that 
\[
    g(N) \leq \frac{\Gamma(1\!-\!2\alpha)}{\Gamma(1\!+\!\alpha)^2} (N\!+\!1\!+\!\alpha)^{4\alpha} \left(\log(N) \!+\! \gamma \!+\! \frac{1}{2N\!-\!1}\right) \leq \frac{\Gamma(1\!-\!2\alpha)}{\Gamma(1\!+\!\alpha)^2} e^{4L} \left(1\!+\!\frac{L}{\alpha}\right) \leq \frac{1}{\pi\alpha},
\]
where we used the fact that $\sum_{d=1}^N d^{-1} \leq \log(N) + \gamma + 1/(2N-1) \leq 1+L/\alpha$~\cite{euler} and the last inequality follows from~\cref{eq.L}. Since $g(N)$ is an increasing function of $N$, we know from \cref{eq.L1} that $N < \Nmin$. Moreover, this holds for all $N > 1$ that satisfies $\log (N+1+\alpha) + 1/(2N-1) \leq (1-\gamma) + L/\alpha$. When $\alpha < 0.15$, we see that $\Nmin \geq 2$ by~\cref{eq.L1} and this proves \cref{eq.N}. 
\end{proof} 
\Cref{cor:Final} implies that $\log(\Nmin) = \Omega(\alpha^{-1})$ and by~\cref{thm.negative}, we conclude that $\log( \Nmin ) = \Theta(\alpha^{-1})$ as $\alpha\rightarrow 0$. 

\section*{Acknowledgments} 
We are thankful for many conversations with Anthony Austin, Nick Trefethen, and Kuan Xu regarding unevenly spaced trigonometric interpolation over several years. In private communication, Heather Wilber gave an initial proof of the discrete Kadec-1/4 theorem in December 2017, which we adapted for our purposes. It was Laurent Demanet who brought our attention to the conditioning of NUDFT matrices and we also benefited from Alex Barnett's wisdom on the subject. The research direction became more focused after brain storming sessions during Cornell's math REU program in 2021 and we thank Aparna Gupte, Yunan Yang, and Liu Zhang for discussions regarding the MZ inequalities and quadrature rules. 

\bibliographystyle{siam}
\bibliography{references}

\end{document}